\documentclass[10pt,notitlepage]{amsart}
\usepackage[T1]{fontenc}
\usepackage[latin9]{inputenc}
\usepackage{color}
\usepackage[letterpaper]{geometry}
\usepackage{latexsym, amsfonts, amsmath, amssymb, amsthm, cite, enumerate}

\usepackage{graphicx}
\usepackage{dsfont}
\usepackage{mathrsfs}

\newtheorem{theorem}{Theorem}
\newtheorem{lemma}[theorem]{Lemma}
\newtheorem{corollary}[theorem]{Corollary}

\theoremstyle{remark}

\newtheorem{remark}{Remark}

\newtheorem*{problem*}{Problem}

\usepackage{xparse}
\NewDocumentCommand{\floor}{s O{} m}{%
  \IfBooleanTF{#1} 
    {\left\lfloor#3\right\rfloor} 
    {#2\lfloor#3#2\rfloor} 
}

\NewDocumentCommand{\ceil}{s O{} m}{%
  \IfBooleanTF{#1} 
    {\left\lceil#3\right\rceil} 
    {#2\lceil#3#2\rceil} 
}

\newcommand{\rad}{\text{rad}}
\newcommand{\bq}{\mathcal{H} ( \tfrac{1}{q} )}
\newcommand{\bprime}{\mathcal{H} ' ( \tfrac{1}{q} )}
\newcommand{\bdoubleprime}{\mathcal{H} '' ( \tfrac{1}{q} )}
\newcommand{\cq}{ \mathcal{C} ( 1, \tfrac{1}{q} ) }
\newcommand{\cprimew}{ \frac{ \frac{d}{dw} \mathcal{C} (1,w) \rvert_{w=\frac{1}{q}}}{q}}
\newcommand{\cprimex}{ \frac{d}{dx} \mathcal{C} (x,\tfrac{1}{q} ) \rvert_{x=1}}
\newcommand{\sumstar}{\sideset{}{^*}\sum}

\numberwithin{theorem}{section} \numberwithin{equation}{section}

\begin{document}

\title{The fourth moment of quadratic Dirichlet $L$--functions over function fields}
\date{}
\author{Alexandra Florea}
\address{Department of Mathematics, Stanford University, Stanford, CA 94305}
\email{amusat@stanford.edu}
\maketitle

\begin{abstract}
We obtain an asymptotic formula for the fourth moment of quadratic Dirichlet $L$--functions over $\mathbb{F}_q[x]$, as the base field $\mathbb{F}_q$ is fixed and the genus of the family goes to infinity. According to conjectures of Andrade and Keating, we expect the fourth moment to be asymptotic to $q^{2g+1} P(2g+1)$ up to an error of size $o(q^{2g+1})$, where $P$ is a polynomial of degree $10$ with explicit coefficients. We prove an asymptotic formula with the leading three terms, which agrees with the conjectured result.
\end{abstract}

\section{Introduction}
In this paper, we study the symplectic family of $L(s,\chi_D)$, as $D$ ranges over square-free polynomials of degree $2g+1$, with coefficients in a fixed field $\mathbb{F}_q[x]$. We obtain an asymptotic formula for the fourth moment of this family of $L$--functions at the critical point, with some of the secondary main terms, as $g \to \infty$ and $q$ is fixed. Specifically, we prove the following.
\begin{theorem}
Let $q$ be a prime with $q \equiv 1 \pmod 4$. Then
\begin{equation}
\sum_{D \in \mathcal{H}_{2g+1}} L ( \tfrac{1}{2},\chi_D)^4 = q^{2g+1} (a_{10}g^{10}+a_9 g^9+a_8 g^8) + O(q^{2g+1}g^{7+\frac{1}{2}+\epsilon}),
\end{equation}
where the sum above is over monic, square-free polynomials of degree $2g+1$ in $\mathbb{F}_q[x]$, and the coefficients $a_{10},a_9,a_8$ are arithmetic factors which can be written down explicitly (see the Appendix, formulas \eqref{coeff1},\eqref{coeff2},\eqref{coeff3}).
\label{fourth}
\end{theorem}
Computing moments in families of $L$--functions is a problem which goes back to Hardy and Littlewood \cite{hl}. If we let
$$ M_k(T) = \int_0^T \left| \zeta ( \tfrac{1}{2} +it )\right|^{2k}  dt,$$
then Hardy and Littlewood showed that $M_1(T) \sim T \log T$, and Ingham \cite{ingham} computed the second moment to be $M_2(T) \sim \frac{1}{2 \pi^2} T (\log T)^4$. In general, it is conjectured that 
$$M_k(T) \sim C_k T (\log T)^{k^2},$$ for some constant $C_k$, whose precise value was predicted by Keating and Snaith \cite{ksnaith}, using analogies with random matrix theory. While no moment higher than $2$ has been computed so far, Soundararajan obtained almost sharp upper bounds, conditional on GRH. More precisely, he showed that $M_k(T) \ll T (\log T)^{k^2+\epsilon},$ for any $\epsilon>0$. Building on this work, Harper \cite{harper} obtained upper bounds of the correct order of magnitude for moments of the Riemann-zeta function, by removing the $\epsilon$ on the power of $\log T$.

One can look at other families of $L$--functions also. In this paper, we will focus on the family of quadratic Dirichlet $L$--functions. As for the Riemann-zeta function, one can only compute a few small moments. Jutila \cite{jutila} obtained asymptotics for the first and second moment of this family. He showed that
 $$\sumstar_{0 < d \leq D} L \Big(\tfrac{1}{2}, \chi_d \Big) \sim C_1 D \log D,$$  where the sum above is over fundamental discriminants, and that
 $$ \sumstar_{0 < d \leq D} L \Big(\tfrac{1}{2}, \chi_d \Big)^2 \sim C_2 D (\log D)^{3},$$  where the constants $C_1$ and $C_2$ can be written down explicitly. 
Soundararajan \cite{sound} computed a secondary main term for the second moment, and also obtained an asymptotic for the third moment. Generally, it is conjectured that
 $$\sumstar_{0 < d \leq D} L \Big(\tfrac{1}{2}, \chi_d \Big)^k \sim C_k D (\log D)^{\frac{k(k+1)}{2}},$$  and the precise value of $C_k$ follows from work of Keating and Snaith \cite{ksnaith}, again using random matrix theory. Conrey, Farmer, Keating, Rubinstein and Snaith \cite{cfkrs} refined this conjecture, and their recipe predicts that the $k^{\text{th}}$ moment above should be asymptotic to $D P_k(\log D)$, where $P_k$ is a polynomial of degree $\frac{k(k+1)}{2}.$ For $k \leq 3$, the computed moments match the answers predicted by the recipe. An alternative approach to computing moments has been proposed by Diaconu, Goldfeld and Hoffstein \cite{dgh}, using multiple Dirichlet series. Their method allows them to compute the cubic moment of $L(1/2,\chi_d)$ and further predicts the existence of infinitely many lower order terms for the fourth moment of this family of $L$-functions, of size $X^{\frac{j+1}{2j}+\epsilon}$, for $j \geq 2$.

 The fourth moment of this family has not been explicitly computed; however, this problem is similar in difficulty with computing the second moment of the orthogonal family of quadratic twists of modular forms. Under GRH, Soundararajan and Young \cite{soundyoung} obtained an asymptotic formula with the leading main term. We are led to believe that using the same circle of ideas, under GRH, one could maybe obtain the leading term for the fourth moment of the family we are interested in.  However, this seems to be right at the edge of what can be achieved for this family of $L$--functions, and it has not been done so far.
 
 Here, we consider moments of the symplectic family of quadratic Dirichlet $L$--functions in the function field setting. We are interested in computing 
 
 \begin{equation}
  \sum_{D \in \mathcal{H}_{2g+1}} L ( \tfrac{1}{2},\chi_D)^k, \label{mom4}
  \end{equation} where $\mathcal{H}_{2g+1}$ denotes the hyperelliptic ensemble of monic, square-free polynomials of degree $2g+1$ with coefficients in $\mathbb{F}_q[x]$, as $|D|:=q^{\deg(D)} = q^{2g+1} \to \infty$. Then we can consider two limits: the limit $q \to \infty$ (and $g$ fixed), or $g \to \infty$ (and $q$ fixed). In the former case, the problem is solved by using the equidistribution results of Katz and Sarnak \cite{katzsarnak}, \cite{katzsarnak2}. As $q \to \infty$, they showed that the Frobenii classes become equidistributed in the group $\text{USp}(2g)$, so computing the moment reduces to computing a matrix integral over $\text{USp}(2g)$, which was done by Keating and Snaith \cite{ksnaith}. Note that in the $q \to \infty$ regime, Bucur and Diaconu \cite{bucurdiaconu} obtained an asymptotic formula for $\displaystyle \sum_{\substack{D \text{ monic} \\ \deg(D) =2g}} L (1/2,\chi_D)^4$, using multiple Dirichlet series.
 
 Hence we concentrate on the limit $g \to \infty$ (with $q$ fixed), which is more similar to the original number field problem. The first moment was computed by Andrade and Keating \cite{keatingandrade}, and their answer is similar to the number field asymptotic. Specifically, they proved that
 $$ \sum_{D \in \mathcal{H}_{2g+1}} L ( \tfrac{1}{2},\chi_D) \sim |D|  P_1(\log_q |D|),$$ where $P_1$ is an explicit linear polynomial whose coefficients are arithmetic terms. A secondary main term of size $|D|^{\frac{1}{3}} \log_q |D|$ was identified in \cite{aflorea}. The second and third moments of this family of $L$--functions were computed in \cite{aflorea2}.
 
 Following the recipe in \cite{cfkrs}, Andrade and Keating \cite{conjectures} conjectured asymptotic formulas for the integral moments of the family of quadratic Dirichlet $L$--functions in function fields. Specifically, they conjectured that
 $$ \sum_{D \in \mathcal{H}_{2g+1}} L ( \tfrac{1}{2},\chi_D)^k = q^{2g+1}( P_k(2g+1) + o(1)),$$ where $P_k$ is a polynomial of degree $\frac{k(k+1)}{2}$, with explicit coefficients. Rubinstein and Wu \cite{rubinstein} provided numerical computations which support Andrade and Keating's conjecture. 
 
 For the fourth moment, the conjecture states that
\begin{equation}
 \sum_{D \in \mathcal{H}_{2g+1}} L \big( \tfrac{1}{2}, \chi_D \big)^4 =  \sum_{D \in \mathcal{H}_{2g+1}} Q  (2g+1)(1+ o(1)), \label{conjmoment}
\end{equation}
where $Q$ is a polynomial of degree $10$ given by 
\begin{equation}
 Q(x) = \frac{ 2^k}{k!} \frac{1}{(2 \pi i)^4} \oint \ldots \oint \frac{G(z_1, \ldots, z_4) \Delta(z_1^2, \ldots, z_4^2)^2}{ \prod_{j=1}^4 z_j^{7}} q^{\frac{x}{2} \sum_{j=1}^4 z_j } \, dz_1 \ldots dz_4,
 \label{qx}
 \end{equation} and
$$G(z_1, \ldots,z_4) = A \left( z_1,\ldots,z_4 \right) \prod_{j=1}^4 X \left( \frac{1}{2} + z_j \right)^{-1/2} \prod_{1 \leq i \leq j \leq 4} \zeta_q(1+z_i+z_j).$$
In the above,
$$X(s)= q^{-1/2+s},$$ and
\begin{align}
 A \left(  z_1, \ldots, z_4 \right) &= \prod_{P} \prod_{1 \leq i \leq j \leq 4} \left( 1- \frac{1}{|P|^{1+z_i+z_j}} \right) \nonumber \\
 & \times \left( \frac{1}{2} \left( \prod_{j=1}^4 \left(1- \frac{1}{|P|^{1/2+z_j}} \right)^{-1} +  \prod_{j=1}^4 \left(1+ \frac{1}{|P|^{1/2+z_j}} \right)^{-1} \right) +\frac{1}{|P|} \right) \left( 1+ \frac{1}{|P|} \right)^{-1}. \label{az}
\end{align}
We obtain an asymptotic formula with the leading three terms, and check that the answer matches the above conjecture.


\section{Background and tools}
\subsection{$L$--functions over function fields}
Here we gather some basic facts about $L$--functions in function fields. Many of the proofs can be found in \cite{rosen}.

Throughout the paper, for simplicity, we will take $q$ to be a prime with $q \equiv 1 \pmod 4$. For $f$ a polynomial in $\mathbb{F}_q[x]$, its degree will be denoted by $d(f)$. The set of monic polynomials of degree $n$ is denoted by $\mathcal{M}_n,$ the set of monic polynomials of degree less than or equal to $n$ by $\mathcal{M}_{\leq n}$, and $\mathcal{H}_n$ denotes the set of monic, square-free polynomials of degree $n$. The symbol $P$ will stand for a monic, irreducible polynomial of degree $n$. 
Note that
$|\mathcal{M}_n|=q^n$, and for $n \geq 1$, $| \mathcal{H}_n| = q^{n-1}(q-1)$. Let $\pi_q(n)$ denote the number of monic, irreducible polynomials of degree $n$. Then the Prime Polynomial Theorem states that
\begin{equation}
\pi_q(n) = \frac{q^n}{n}+O \Big(  \frac{q^{n/2}}{n} \Big). \label{ppt}
\end{equation}
For a polynomial $f$ in $\mathbb{F}_q[x]$, let $|f|:=q^{d(f)}$ denote the norm of $f$. For $\Re(s)>1$, the zeta-function of $\mathbb{F}_q[x]$ is defined by
$$ \zeta_q(s) = \sum_{f \text{ monic}} \frac{1}{|f|^s}= \prod_P (1-|P|^{-s})^{-1}.$$
Since there are $q^n$ monic polynomials of degree $n$, one can show that $$\zeta_q(s) = \frac{1}{1-q^{1-s}},$$ and this provides an analytic continuation of the zeta-function to the complex plane, with a simple pole at $s=1$. We will often use the change of variables $u=q^{-s}$. Then
$$\mathcal{Z}(u) = \zeta_q(s) = \sum_{f \text{ monic}} u^{d(f)} = \prod_P (1-u^{d(P)})^{-1} =  \frac{1}{1-qu} .$$
The M\"{o}bius function $\mu$ is defined as usual by $\mu(f)=(-1)^{\omega(f)}$ if $f$ is a square-free polynomial and where $\omega(f) =\sum_{P|f} 1$, and $0$ otherwise. We also define the von-Mangoldt function as
$$ \Lambda(f) = 
\begin{cases}
d(P) & \mbox{ if } f=c P^k, c \in \mathbb{F}_q^{\times} \\
0 & \mbox{ otherwise.}
\end{cases}
$$
Now for $P$ a monic irreducible polynomial, define the quadratic character $\Big( \frac
{f}{P}\Big)$ by
$$ \Big( \frac{f}{P}\Big)= 
\begin{cases}
1 & \mbox{ if }  f \text{ is a square} \pmod P, P \nmid f \\
-1& \mbox{ if }  f \text{ is not a square} \pmod P, P \nmid f \\
0 & \mbox{ if } P|f.
\end{cases}
$$
For $D$ a polynomial in $\mathbb{F}_q[x]$, the symbol $ \Big(  \frac{ \cdot}{D} \Big)$ is defined by extending the residue above multiplicatively. The quadratic reciprocity law states that for $A,B$ non-zero, relatively prime monic polynomials
$$ \Big(  \frac{A}{B} \Big)= \Big(  \frac{B}{A } \Big) (-1)^{\frac{ q-1}{2} d(A) d(B)}.$$
Since $q \equiv 1 \pmod 4$, one has $ \Big(  \frac{A}{B} \Big)= \Big(  \frac{B}{A } \Big).$ The quadratic character $\chi_D$ is defined by $$\chi_D(f) = \Big(  \frac{D}{f} \Big).$$ 
Now the $L$--function associated to the quadratic character $\chi_D$ is defined by
$$L(s,\chi_D) = \sum_{f \text{ monic}} \frac{\chi_D(f)}{|f|^s} = \prod_P (1-\chi_D(P) |P|^{-s})^{-1}.$$
Similarly as before, with the change of variables $u=q^{-s}$, one has
$$ \mathcal{L}(u,\chi_D) = \sum_{f \text{ monic}} \chi_D(f) u^{d(f)} = \prod_P (1-\chi_D(P) u^{d(P)})^{-1}.$$
When $D$ is a non-square polynomial, since when $d(f) \geq d(D)$
$$\sum_{f \in \mathcal{M}_m} \chi_D(f)=0,$$ it follows that $\mathcal{L}(u,\chi_D)$ is a polynomial of degree at most $d(D)-1$. 
From now on, $D$ will be a monic, square-free polynomial of odd degree $2g+1$. We have
\begin{equation}
\mathcal{L}(u,\chi_D) = \prod_{i=1}^{2g} (1-u \sqrt{q} \alpha_j) .
\label{produs}
\end{equation}
The Riemann hypothesis, proven by Weil \cite{weil}, states that $|\alpha_j|=1$, hence we can write $\alpha_j = e^{2 \pi i \theta_j}$, with $\theta_j \in \mathbb{R} / \mathbb{Z}.$ Moreover, the $L$--function satisfies the functional equation
$$ \mathcal{L}(u,\chi_D) = (qu^2)^g \mathcal{L} \Big(  \frac{1}{qu},\chi_D \Big).$$
One can define the completed $L$--function in the following way. Let $X_D(s) = |D|^{\frac{1}{2}-s} X(s),$ where $X(s) = q^{s-\frac{1}{2}}.$ Let 
\begin{equation}
\Lambda(s,\chi_D) = L(s,\chi_D) X_D(s)^{-\frac{1}{2}}.
\label{completed}
\end{equation}
Then the completed $L$--function above satisfies the symmetric functional equation
\begin{equation}
\Lambda(s,\chi_D) =\Lambda(1-s,\chi_D). \label{fecompleted}
\end{equation} 

\subsection{Preliminary lemmas}
\begin{lemma} \label{afe}
We have the following ``approximate functional equation'':
$$ \sum_{D \in \mathcal{H}_{2g+1}} L ( \tfrac{1}{2}, \chi_D )^4= \sum_{f \in \mathcal{M}_{\leq 4g}} \frac{d_4(f) \chi_D(f)}{\sqrt{|f|}} +  \sum_{f \in \mathcal{M}_{\leq 4g-1}} \frac{d_4(f) \chi_D(f)}{\sqrt{|f|}},$$ where $d_4$ denotes the $4^{\text{th}}$ divisor function (i.e.: $d_4(f) = \sum_{f_1f_2f_3f_4= f}1$).
\end{lemma}
\begin{proof}
See Lemma $2.1$ in \cite{aflorea2}.
\end{proof}
The following lemma allows us to express sums over square-free polynomials in terms of sums over monic polynomials. For a proof of this, see Lemma $2.1$ in \cite{aflorea}. 
\begin{lemma}
For $f$ a monic polynomial in $\mathbb{F}_q[x]$, we have that
$$ \sum_{D \in \mathcal{H}_{2g+1}} \chi_D(f)= \sum_{C | f^{\infty}} \sum_{h \in \mathcal{M}_{2g+1-2 d(C)}} \chi_f(h) - q \sum_{C | f^{\infty}} \sum_{h \in \mathcal{M}_{2g-1-2 d(C)}} \chi_f(h),$$
where the first summation is over monic polynomials $C$ whose prime factors are among the prime factors of $f$. \label{sumd} \end{lemma}
Now recall the exponential function introduced in \cite{hayes}. For $a \in \mathbb{F}_q((\frac{1}{x}))$, let
$$e(a) = e^{\frac{2 \pi i a_1}{q}},$$ where $a_1$ is the coefficient of $1/x$ in the Laurent expansion of $a$. 
Then the generalized Gauss sum is defined by
$$
G(u,\chi) = \displaystyle \sum_{V \pmod f} \chi(V) e \left(\frac{uV}{f}\right).
$$

The following lemma is the analog of the Poisson summation formula in function fields, and the proof can be found in \cite{aflorea}.
\begin{lemma}
Let $f$ be a monic polynomial of degree $n$ in $\mathbb{F}_q[x]$ and let $m$ be a positive integer. If the degree $n$ of $f$ is even, then
\begin{equation}
\sum_{h \in \mathcal{M}_m} \chi_f(g) = \frac{q^m}{|f|} \left[ G(0,\chi_f) + (q-1) \sum_{V \in \mathcal{M}_{\leq n-m-2} } G(V,\chi_f) - \sum_{V \in \mathcal{M}_{n-m-1}} G(V,\chi_f) \right]. \label{even}
\end{equation}
If $n$ is odd, then
\begin{equation}
 \sum_{h \in \mathcal{M}_m} \chi_f(h) = \frac{q^m}{|f|} \sqrt{q} \sum_{V \in \mathcal{M}_{n-m-1}} G(V,\chi_f).  \label{odd} \end{equation}
\label{poissonmonic}
\end{lemma}
We also need the following.
\begin{lemma}
Suppose that $q \equiv 1 \pmod 4$. Then 
\begin{enumerate}
\item If $(f,h)=1$, then $G(V, \chi_{fh})= G(V, \chi_f) G(V,\chi_h)$.
\item Write $V= V_1 P^{\alpha}$ where $P \nmid V_1$.
Then 
 $$G(V , \chi_{P^i})= 
\begin{cases}
0 & \mbox{if }  i \leq \alpha \text{ and } i \text{ odd} \\
\phi(P^i) & \mbox{if }  i \leq \alpha \text{ and } i \text{ even} \\
-|P|^{i-1} & \mbox{if }  i= \alpha+1 \text{ and } i \text{ even} \\
\left( \frac{V_1}{P} \right) |P|^{i-1} |P|^{1/2} & \mbox{if } i = \alpha+1 \text{ and } i \text{ odd} \\
0 & \mbox{if } i \geq 2+ \alpha .
\end{cases}$$ 
\end{enumerate} \label{computeg}
\end{lemma}
Using the previous two lemmas, one can prove an analog of the Polya-Vinogradov inequality.
\begin{lemma} \label{pv}
For $f$ a non-square polynomial and $m < d(f)$, we have the following inequality:
$$ \sum_{h \in \mathcal{M}_m} \chi_f(h) \ll \sqrt{|f|}.$$
\end{lemma}
\begin{proof}
We first prove the inequality for $f$ a square-free polynomial. By Lemma \ref{computeg}, $G(V,\chi_f)= \chi_f(V) \sqrt{|f|}.$ Then using the Poisson summation formula and trivially bounding the sum over $V$ in Lemma \ref{poissonmonic} gives the desired upper bound.

Now assume that $f$ is not necessarily square-free. Write $f= Df_1^2$, where $D$ is a square-free polynomial. Then
\begin{align*}
\sum_{h \in \mathcal{M}_m} \chi_f(h) = \sum_{\substack{h \in \mathcal{M}_m \\ (h,f_1)=1}} \chi_D(h) = \sum_{B | f_1} \mu(B) \sum_{\substack{ h \in \mathcal{M}_m \\ B|h}} \chi_D(h) = \sum_{B | f_1} \mu(B) \chi_D(B) \sum_{h_1 \in \mathcal{M}_{m - d(B)}} \chi_D(h_1) \ll \sqrt{|f|},
\end{align*} where the last inequality follows by using the upper bound proved previously on the character sum.
\end{proof}
We also have the following explicit formula in function fields, relating sums over zeros to sums over irreducibles. The proof can be found in \cite{rudnickfaifman} (see Lemma $2.2$).
\begin{lemma}
\label{explicit}
Let $h(\theta) = \sum_{|k| \leq N} \hat{h}(k) e(k \theta)$ be a real valued, even trigonometric polynomial. For $D$ a square-free polynomial of degree $2g+1$, we have
$$ \sum_{j=1}^{2g} h( \theta_j) = 2g \int_0^1 h(\theta) \, d\theta -2 \sum_f \hat{h}(d(f)) \frac{\chi_D(f) \Lambda(f)}{\sqrt{|f|}}.$$
\end{lemma}

\subsection{Upper bounds on moments of $L$--functions}

One of the main ingredients in the proof of Theorem \ref{fourth} is obtaining upper bounds for moments of quadratic Dirichlet $L$--functions. This follows from work of Soundararajan \cite{soundupperbounds} on upper bounds for moments of the Riemann-zeta function, conditionally on $RH$. Since the Riemann hypothesis was proven in function fields, the upper bounds we obtain are unconditional. We note also the work of Chandee \cite{chandeeriemann} on shifted moments for the Riemann-zeta function and the work of Harper \cite{harper} for getting sharp upper bounds. 

To obtain Theorem \ref{fourth}, we will also need upper bounds for a product of shifted moments of quadratic Dirichlet $L$--functions, when the number of $L$--functions grows (slowly) with the genus of the family. Proving this is a bit more delicate, and relies on using ideas as in the work of Chandee and Soundararajan \cite{soundfai}, or Carneiro and Chandee \cite{faicarneiro}. We will elaborate on this in section \ref{upper}.

When considering upper bounds for a single shifted $L$--function, we have the following result.
\begin{theorem}
Let $u=e^{i \theta}$, with $\theta \in [0, \pi)$. Then for every positive $k$ and any $\epsilon>0$, 
$$\sum_{D \in \mathcal{H}_{2g+1}} \left| \mathcal{L} \left( \frac{u}{\sqrt{q}}, \chi_D \right) \right|^k \ll q^{2g+1} g^{\epsilon} \text{exp} \left( k \mathcal{M}(u,g)+\frac{k^2}{2} \mathcal{V}(u,g) \right),$$ where $ \mathcal{M}(u,g) = \frac{1}{2} \log \left( \min \{ g, \frac{1}{2 \theta} \} \right)$ and $\mathcal{V}(u,g) = \mathcal{M}(u,g) + \frac{\log(g)}{2}.$
\label{upperbound}
\end{theorem}
The corollary immediately follows from the theorem above.
\begin{corollary}
With the same notation as before,
$$\sum_{D \in \mathcal{H}_{2g+1}} \left| \mathcal{L} \left( \frac{u}{\sqrt{q}}, \chi_D \right) \right|^4 \ll q^{2g+1} g^{4+\epsilon} \left( \min \Big \{g, \frac{1}{ \theta}\Big  \} \right)^6.$$ \label{cor4}
\end{corollary}
\begin{remark}
Note that when $\theta$ is close to $0$ (so when we are evaluating the $L$--function close to the critical point), the family behaves like a family with symplectic symmetry. As we move away from the critical point, the symmetry type becomes unitary. 
\end{remark} 
We postpone the proof of the theorem to section ~\ref{upper}.
\subsection{Outline of the proof}
We start similarly as for the lower moments of quadratic Dirichlet $L$--functions. We use the functional equation for $L(1/2,\chi_D)^4$, and then we manipulate the sum over square-free polynomials $D$ to obtain sums involving monic polynomials. When the sums are ``long'', we use the Poisson summation formula in function fields to obtain ``shorter'' sums. When the sums are already ``short'', then we go back to a sum involving square-free polynomials, and then use upper bounds for moments of $L$--functions to show that this term is negligible.

After using the Poisson summation formula for the long sums, there will be a main term corresponding to $V=0$, where $V$ is the dual parameter in the Poisson formula. We evaluate this term in section \ref{maintermsec} and find that it is of size $q^{2g+1} g^{10}$. The term corresponding to $V$ a square polynomial is also of size $q^{2g+1} g^{10}$, and we evaluate it exactly in section \ref{secmain}.

Evaluating the term coming from $V$ a non-square is the most subtle part of the argument. Bounding this term amounts to bounding a shifted fourth moment, integrated along a circle. The key idea is noticing that the family of $L$--functions $\mathcal{L}(u,\chi_D)$ behaves differently as $u$ moves along the circle of radius $1/\sqrt{q}$. When $u$ is ``close'' (on an arc of angle $1/g$) to the critical point $1/\sqrt{q}$, the family behaves like a family with symplectic symmetry, hence we expect a power of $g^{10}$ from the integral on this arc. As $u$ moves further away from the critical point, the power of $g$ decreases. Then using upper bounds on moments of $L$--functions, we can bound the contribution of $V$ non-square by a smaller power of $g$ (more precisely, we bound this term by $q^{2g+1} g^{9+\epsilon}$). Hence this provides an asymptotic formula for the fourth moment with the main term of size $q^{2g+1} g^{10}.$ 

To get a few of the other lower order terms, we use a recursive argument. We consider the sum over polynomials $f$ (from the approximate functional equation), truncated at $4g - \alpha$ (where $\alpha$ is on the scale of $ \log g$). By similar techniques as before, we can get an asymptotic for this term, with an error of size $o(q^{2g+1})$. For the tail, we use Perron's formula to express it as an integral of a shifted fourth moment over a circle around the origin. Similarly as before, we split the circle into two arcs: one ``small'' arc around the critical point, and its complement. On the small arc, we will use the asymptotic formula we already have (with the error of size $g^{9+\epsilon}$), and for the complement we use again upper bounds on moments of $L$--functions. By plugging in the main term from the asymptotic formula and integrating it along the ``small'' arc, we show that the terms of size $g^9 \alpha, g^8 \alpha^2$ and $g^8 \alpha$ coming from the tail and the truncated sums cancel out, leaving us with the lower order main terms of size $g^9$ and $g^8$, and an improved error of size $g^{7+\frac{1}{2}+\epsilon}$. This argument also shows that the contribution from non-square polynomials $V$ (which we initially bounded by $g^{9+\epsilon}$) is actually bounded by $g^{7+\frac{1}{2}+\epsilon}$. 

By repeating this argument and carefully matching up terms, one could get an asymptotic formula with all the lower order terms down to $g^5$. Surpassing the $g^{4+\epsilon}$ error bound seems to be a challenge, since the $g^{4+\epsilon}$ error comes from using upper bounds on shifted moments, when the point we consider is far from the critical point. We don't have a way of obtaining an asymptotic formula in this case.

In the Appendix, we will prove various easy identities, inequalities and asymptotic formulas for sums involving trigonometric functions. We have decided to include them here for the sake of completeness. Moreover, we write down explicitly the coefficients in Theorem \ref{fourth}, and briefly show (by direct computation) that they match the conjectured answer.

\section{Setup of the problem}
Using the functional equation in Lemma \ref{afe} and Lemma \ref{sumd}, we have that
$$ \sum_{D \in \mathcal{H}_{2g+1}} L \left( \frac{1}{2}, \chi_D \right)^4= S_{4g}+S_{4g-1},$$
where
\begin{equation}
S_{4g} = \sum_{f \in \mathcal{M}_{\leq 4g}} \frac{d_4(f)}{ \sqrt{|f|}} \sum_{\substack{C \in \mathcal{M}_{\leq g} \\ C | f^{\infty} }} \sum_{h \in \mathcal{M}_{2g+1-2d(C)}} \chi_f(h)- q \sum_{f \in \mathcal{M}_{\leq 4g}} \frac{d_4(f)}{ \sqrt{|f|}} \sum_{\substack{C \in \mathcal{M}_{\leq g-1} \\ C | f^{\infty} }} \sum_{h \in \mathcal{M}_{2g-1-2d(C)}} \chi_f(h). \label{s4g}
\end{equation}
We similarly define $S_{4g-1}$. In equation ~\ref{s4g} above, let $S_{4g,1}$ denote the term with $d(C) \leq y$, where $y$ is a parameter we will choose later, and let $S_{4g,2}$ denote the term with $d(C) >y$. We treat $S_{4g,1}$ and $S_{4g,2}$ differently. For $S_{4g,1}$ we will use the Poisson summation formula, and we will bound $S_{4g,2}$ as follows.
\begin{lemma}
Using the previous notation, we have that
$$S_{4g,2} \ll q^{2g-3y/2} g^{10+\epsilon} y^3.$$
\label{tailc}
\end{lemma}
\begin{proof}
It is enough to bound the following term
$$S:= \sum_{f \in \mathcal{M}_{\leq 4g}} \frac{d_4(f)}{\sqrt{|f|}}  \sum_{\substack{y<d(C) \leq g \\ C | f^{\infty}}} \sum_{h \in \mathcal{M}_{2g+1-2d(C)}} \chi_f(h).$$
Since $C | f^{\infty}$, we write $f= \rad(C) W$, where $W \in \mathcal{M}$, and let $h= B^2 A$, where $A,B$ are monic and $A$ is a square-free polynomial. We rewrite
\begin{equation}
S= \sum_{y<d(C) \leq g} \frac{1}{\sqrt{|\rad(C)|}} \sum_{\substack{B \in \mathcal{M}_{\leq g-d(C)} \\ (B,C)=1}}  \sum_{A \in \mathcal{H}_{2g+1-2d(C)-2d(B)}} \chi_A(\rad(C)) \sum_{\substack{W \in \mathcal{M}_{\leq 4g-d(\rad(C))} \\ (W,B)=1}} \frac{d_4(W \rad(C))}{\sqrt{|W|}} \chi_A(W). \label{form}
\end{equation}
For $R$ a square-free polynomial, and $A,B$ monic polynomials, let
$$ \mathcal{C}_{R,A,B}(u) = \sum_{\substack{ W \in \mathcal{M} \\ (W,B)=1}} d_4(WR) \chi_A(W) u^{d(W)}.$$ Then
\begin{align*}
& \mathcal{C}_{R,A,B}(u)  = \prod_{\substack{P \nmid B \\ P \nmid R}}  \left( \sum_{i=0}^{\infty} u^{id(P)} \chi_A(P^i) d_4(P^i) \right) \prod_{\substack{ P \nmid B \\ P | R}} \left( \sum_{i=0}^{\infty} u^{id(P)} \chi_A(P^i) d_4(P^{i+1}) \right) \\
 &= \mathcal{L}(u,\chi_A)^4 \prod_{\substack{P \nmid B \\ P|R}} \left(  \sum_{i=0}^{\infty} u^{id(P)} \chi_A(P^i) d_4(P^{i+1}) \right) \left( \sum_{i=0}^{\infty} u^{id(P)} \chi_A(P^i) d_4(P^i) \right) ^{-1} \prod_{\substack{P|B \\ P \nmid A}} \left( \sum_{i=0}^{\infty} u^{id(P)} \chi_A(P^i) d_4(P^i) \right) ^{-1}.
 \end{align*}
 Let $\mathcal{C}_{P,R,A,B}(u)$ denote the $P$-factors above. Using the above and Perron's formula in ~\eqref{form}, we have 
 $$ S = \frac{1}{2 \pi i} \oint_{\gamma} \sum_{y<d(C) \leq g} \frac{1}{\sqrt{|\rad(C)|}} \sum_{\substack{B \in \mathcal{M}_{\leq g-d(C)} \\ (B,C)=1}}  \sum_{A \in \mathcal{H}_{2g+1-2d(C)-2d(B)}} \chi_A(\rad(C)) \frac{ \mathcal{L}(u,\chi_A)^4 \prod_P \mathcal{C}_{P,\rad(C),A,B}(u)}{(1-\sqrt{q}u)(\sqrt{q}u)^{4g-d(\rad(C))}} \, \frac{du}{u} ,$$ where $\gamma$ is a circle around the origin of radius less than $1/\sqrt{q}$. The integrand has a pole at $u=1/\sqrt{q}$, so
 $$ |S| \ll  \sum_{y<d(C) \leq g} \frac{d_4(\rad(C))}{\sqrt{|\rad(C)|}} \sum_{\substack{B \in \mathcal{M}_{\leq g-d(C)} \\ (B,C)=1}}  \sum_{A \in \mathcal{H}_{2g+1-2d(C)-2d(B)}} \left| \mathcal{L} \left( \frac{1}{\sqrt{q}} ,\chi_A \right)^4 \right|.$$
 Using the upper bound in Theorem ~\ref{upperbound}, it follows that
 $$|S| \ll  q^{2g+1} \sum_{y<d(C) \leq g} \frac{d_4(\rad(C))}{\sqrt{|\rad(C)|}|C|^2} \sum_{\substack{B \in \mathcal{M}_{\leq g-d(C)} \\ (B,C)=1}} \frac{1}{|B|^2} (2g+1-2d(C)-2d(B))^{10+\epsilon}.$$ Since 
 $$ \sum_{C \in \mathcal{M}_i} \frac{ d_4(\rad(C))}{\sqrt{|\rad(C)|}} \ll q^{i/2} i^3,$$ the conclusion now follows.
\end{proof}
Now we rewrite
$$S_{4g,1} = \sum_{f \in \mathcal{M}_{\leq 4g}} \frac{d_4(f)}{\sqrt{|f|}} \sum_{\substack{C \in \mathcal{M}_{\leq y} \\ C | f^{\infty}}} \left(  \sum_{h \in \mathcal{M}_{2g+1-2d(C)}} \chi_f(h)- q  \sum_{h \in \mathcal{M}_{2g-1-2d(C)}} \chi_f(h) \right),$$ and use the Poisson summation formula for the sum over $h$. Let $S_{4g,\text{e}}$ denote the sum over polynomials $f$ of even degree, and $S_{4g,\text{o}}$ the sum over polynomials $f$ of odd degree. When $d(f)$ is even, let $M_{4g}$ be the term with $V=0$, where $V$ is the dual parameter from the Poisson summation formula, and $S_{4g,\text{e}}(V \neq 0)$ the term with $V \neq 0$. Then $S_{4g,\text{e}} = M_{4g}+S_{4g,\text{e}}(V \neq 0)$. Since $G(0,\chi_f)$ is nonzero if and only if $f$ is a square (and $G(0,\chi_f) = \phi(f)$ in the case), we have
\begin{equation}
 M_{4g} = q^{2g+1} \left(1-\frac{1}{q} \right) \sum_{\substack{f \in \mathcal{M}_{\leq 4g} \\ f= \square}} \frac{d_4(f)}{|f|^{\frac{3}{2}}} \phi(f) \sum_{\substack{C \in \mathcal{M}_{\leq y} \\ C | f^{\infty}}}   \frac{1}{|C|^2} , \label{m4g}
 \end{equation}
and 
 \begin{align} S_{4g,\text{e}} &(V \neq 0) = q^{2g+1} \sum_{\substack{f \in \mathcal{M}_{\leq 4g} \\ d(f) \text{ even}}} \frac{d_4(f)}{|f|^{\frac{3}{2}}}\sum_{\substack{C \in \mathcal{M}_{\leq y} \\ C | f^{\infty}}} \frac{1}{|C|^2} \bigg[ (q-1) \sum_{ V \in \mathcal{M}_{\leq d(f)-2g-3+2d(C)} } G(V, \chi_f) \nonumber \\
  &- \sum_{V \in \mathcal{M}_{d(f)-2g-2+2d(C)}} G(V,\chi_f)  - \frac{q-1}{q} \sum_{V \in \mathcal{M}_{\leq d(f)-2g-1+2d(C)} } G(V, \chi_f) + \frac{1}{q} \sum_{V \in \mathcal{M}_{d(f)-2g+2d(C)}} G(V,\chi_f) \bigg] .\label{s1e} \end{align}
Now write $S_{4,\text{e}}(V \neq 0) = S_{4g,\text{e}}(V =\square)+ S_{4g,\text{e}}(V \neq \square)$, where $S_{4g,\text{e}}(V =\square)$ is the sum over square polynomials $V$. When $V$ is a square, write $V=l^2$. Then
\begin{align}
S_{4g,\text{e}}(V= \square) &= q^{2g+1}\sum_{\substack{f \in \mathcal{M}_{\leq 4g} \\ d(f) \text{ even}}} \frac{d_4(f)}{|f|^{\frac{3}{2}}}\sum_{\substack{C \in \mathcal{M}_{\leq y} \\ C | f^{\infty}}} \frac{1}{|C|^2} \bigg[ (q-1)  \sum_{ l \in \mathcal{M}_{\leq \frac{d(f)}{2}-g-2+d(C)} } G(l^2, \chi_f) \nonumber \\
&- \sum_{l \in \mathcal{M}_{\frac{d(f)}{2}-g-1+d(C)}} G(l^2,\chi_f) - \frac{q-1}{q}  \sum_{ l \in \mathcal{M}_{\leq \frac{d(f)}{2}-g-1+d(C)} } G(l^2, \chi_f) + \frac{1}{q} \sum_{l \in \mathcal{M}_{\frac{d(f)}{2}-g+d(C)}} G(l^2,\chi_f) \bigg] .\label{s1s} \end{align} 
  Let $S_4(V=\square)= S_{4g,\text{e}}(V= \square)+S_{4g-1,\text{e}}(V= \square)$, where $S_{4g-1,\text{e}}(V= \square)$ is defined similarly as $S_{4g,\text{e}}(V= \square)$. 
  
  Also define $S_{4g}(V \neq \square) = S_{4g,\text{o}} + S_{4g,\text{e}}(V \neq \square)$, where
   \begin{equation}
 S_{4g,\text{o}} = q^{2g+1}  \sqrt{q} \sum_{\substack{f \in \mathcal{M}_{\leq 4g} \\ d(f) \text{ odd}}} \frac{d_4(f)}{|f|^{\frac{3}{2}}} \sum_{\substack{C \in \mathcal{M}_{\leq y} \\ C | f^{\infty}}} \frac{1}{|C|^2}  \bigg[ \sum_{V \in \mathcal{M}_{d(f)-2g-2+2d(C)}} G(V, \chi_f)  - \frac{1}{q} \sum_{V \in \mathcal{M}_{d(f)-2g+2d(C)}} G(V, \chi_f) \bigg]. \label{odd2} \end{equation}
 \section{Main term} \label{maintermsec}
 In this section, we will evaluate the main term ~\eqref{m4g}. In equation ~\eqref{m4g}, write $f=l^2$ and since $\zeta_q(2) = (1-q^{-1})^{-1}$, we have
 $$ M_{4g} = \frac{q^{2g+1}}{\zeta_q(2)} \sum_{l \in \mathcal{M}_{\leq 2g}} \frac{d_4(l^2) \phi(l^2)}{|l|^3}  \sum_{\substack{C \in \mathcal{M}_{\leq y} \\ C | l^{\infty}}} \frac{1}{|C|^2}.$$
 Note that
 \begin{equation}
 \sum_{\substack{C \in \mathcal{M}_i \\ C | l^{\infty}}} \frac{1}{|C|^2} = \frac{1}{2 \pi i} \oint_{|u|=r_1} \frac{1}{q^{2i} u^{i+1} \prod_{P|l}(1-u^{d(P)})} \, du,
 \label{sumc}
 \end{equation}
 where $r_1<1$. Let
 $$ \mathcal{A}(w,u) = \sum_{l \in \mathcal{M}} \frac{d_4(l^2) \prod_{P|l} (1 - \frac{1}{|P|})}{\prod_{P|l}(1-u^{d(P)})} w^{d(l)}.$$
 Then
 $$\mathcal{A}(w,u) = \mathcal{Z}(w)^{10} \mathcal{H}(w,u),$$
 where 
 \begin{equation}
  \mathcal{H}(w,u) = \prod_P (1-w^{d(P)})^{10} \Big( 1+ \frac{(|P|-1)w^{d(P)} (10-5w^{d(P)}+4w^{2d(P)}-w^{3d(P)}}{|P|(1-u^{d(P)})(1-w^{d(P)})^4} \Big).
  \label{hw}
  \end{equation}
 Note that $\mathcal{H}(w,u)$ converges absolutely for $|wu|< \frac{1}{q} , |w|< \frac{1}{\sqrt{q}}, |u|<1$. Using Perron's formula twice, we get that
 $$M_{4g} =  \frac{q^{2g+1}}{\zeta_q(2)} \frac{1}{(2 \pi i)^2} \oint_{|w|=\frac{1}{q^{1+\epsilon}}} \oint_{|u|=\frac{1}{q^{2+\epsilon}}} \frac{\mathcal{H}(w,u)}{(1-qw)^{11}(qw)^{2g} (1-q^2u)(q^2u)^y} \, \frac{du}{u} \, \frac{dw}{w}.$$  Now we shift the contour over $w$ to a circle of radius $|w|= R_2 = \frac{1}{q^{\frac{1}{2}+\epsilon}},$ and we encounter a pole of order $11$ at $w=\frac{1}{q}$. We evaluate the residue at $w=\frac{1}{q}$, and get
 \begin{align*}
 M_{4g} &=  \frac{q^{2g+1}}{\zeta_q(2)} \Big[ \frac{1}{2 \pi i} \oint_{|u|=\frac{1}{q^{2+\epsilon}}} \frac{1}{(1-q^2u)(q^2u)^y} P_u(2g) \, \frac{du}{u} \\
 & +  \frac{1}{(2 \pi i)^2} \oint_{|w|=\frac{1}{q^{\frac{1}{2}+\epsilon}}} \oint_{|u|=\frac{1}{q^{2+\epsilon}}} \frac{\mathcal{H}(w,u)}{(1-qw)^{11}(qw)^{2g} (1-q^2u)(q^2u)^y} \, \frac{du}{u} \, \frac{dw}{w} \Big],
 \end{align*}
 where $P_u(x)$ is a polynomial of degree $10$ with coefficients arithmetic factors depending on $u$. Note that the double integral above is of size $O(q^{g(1+\epsilon)}).$ In the first integral, we shift the contour of integration to $|u|= \frac{1}{q^{\epsilon}}$, encountering a pole at $u = \frac{1}{q^2}$ with residue $P_{\frac{1}{q^2}} (2g)$. We rewrite the main term as
 \begin{equation}
 M_{4g} =  \frac{q^{2g+1}}{\zeta_q(2)} P_1(2g+1)+O \Big(q^{2g-(2-\epsilon)y}+q^{g(1+\epsilon)} \Big),
 \label{maint}
 \end{equation}
 where $P_1(2g+1)=P_{\frac{1}{q^2}}(2g).$
\section{Secondary main term}
\label{secmain}
Here we will evaluate the term $S_4(V=\square)$. Recall that $S_4(V=\square) =S_{4g}(V=\square)+S_{4g-1}(V=\square)$, with $S_{4g}(V=\square)$ given by \eqref{s1s}. We'll prove the following.
\begin{lemma}
Using the same notation as before, we have
$$S_4(V=\square) = \frac{q^{2g+1}}{\zeta_q(2)} P_2(2g+1) + O \Big(q^{\frac{3g}{2}(1+\epsilon)} + q^{2g-(2-\frac{\epsilon}{2}) y} \Big),$$ where $P_2(x)$ is a polynomial of degree $10$ whose coefficients can be computed explicitly.
\label{square}
\end{lemma}
Before proving Lemma \ref{square}, we first need the following auxiliary lemma.
\begin{lemma}
For $|u|<1$, let
$$ \mathcal{A}(x,w,u) = \sum_{l \in \mathcal{M}} x^{d(l)} \sum_{f \in \mathcal{M}} \frac{d_4(f) G(l^2,\chi_f)}{ \sqrt{|f|} \prod_{P|f} (1-u^{d(P)})} w^{d(f)}.$$
Then
$$ \mathcal{A}(x,w,u) = \frac{ \mathcal{Z}(w)^4 \mathcal{Z}(x) \mathcal{Z} (qw^2x)^{10} } { \mathcal{Z}(wx)^4 } \mathcal{B}(x,w,u),$$ where
$\mathcal{B}(x,w,u) = \prod_P \mathcal{B}_P(x,w,u),$ with
\begin{align*}
 \mathcal{B}_P(x,w,u) &= \frac{ (1-w^{d(P)})^4 (1-|P|(w^2x)^{d(P)})^6}{ (1-(wx)^{d(P)})^4} \Bigg[ 1+ \frac{1}{1-u^{d(P)}} \Bigg(4 w^{d(P)}-4 (wx)^{d(P)} \\
 &+ 6 |P| (w^2x)^{d(P)}  +4|P| (w^2xu)^{d(P)}- 10 (w^2x)^{d(P)} +4|P|(w^3x)^{d(P)}-4|P|(w^3x^2)^{d(P)}\\
 &+5|P|(w^4x^2)^{d(P)} + |P|^2 (w^4x^2)^{d(P)} - 6|P|^2 (w^4x^2u)^{d(P)}-4|P|^2 (w^6x^3)^{d(P)} \\
 &+4|P|^3 (w^6x^3u)^{d(P)}+|P|^3 (w^8 x^4)^{d(P)}  - |P|^4 (w^8 x^4u)^{d(P)} \Bigg) \Bigg].
 \end{align*}
 Moreover, $\mathcal{B}(x,w,u)$ converges absolutely for $|w|<\frac{1}{\sqrt{q}} , |wu| < \frac{1}{q}, | xwu| < \frac{1}{q}, |qxw^2u| < \frac{1}{q}, |xw| <\frac{1}{ \sqrt{q} }, |q^2 x^2 w^4| < \frac{1}{q}, |q x w^3| < \frac{1}{q}, |qx^2w^3| < \frac{1}{q}.$
\label{auxiliary}
\end{lemma}
\begin{proof}
Since $G(l^2,\chi_f)$ is multiplicative as a function of $f$, we can write
\begin{align*}
 \sum_{f \in \mathcal{M}} \frac{d_4(f) G(l^2,\chi_f)}{ \sqrt{|f|} \prod_{P|f} (1-u^{d(P)})} w^{d(f)} &= \mathcal{Z}(w)^4 \prod_{P | l} (1-w^{d(P)})^4 \left( 1+ \frac{1}{1-u^{d(P)}} \sum_{i=1}^{\infty} \frac{d_4(P^i) G(l^2,\chi_{P^i}) w^{id(P)} } { |P|^{i/2} } \right) \\
 & \prod_{P \nmid l} \Bigg( 1+ \frac{4 (wu)^{d(P)} }{1-u^{d(P)}} +6 w^{2d(P)} - \frac{16 w^{2d(P)} }{1-u^{d(P)}} - 4w^{3d(P)} + \frac{24 w^{3 d(P)}}{1-u^{d(P)}} \\
 & + w^{4 d(P)} - \frac{16 w^{4 d(P)}}{1-u^{d(P)}} +\frac{4w^{5 d(P)}}{1-u^{d(P)}} \Bigg).
 \end{align*}
 Introducing the sum over $l$ and using the multiplicativity of the Euler products give the expression in Lemma \ref{auxiliary}. The absolute convergence of $\mathcal{B}(x,w,u)$ follows directly from the expression for $\mathcal{B}_P(x,w,u)$ by factoring out the appropriate zeta-functions.
\end{proof}

Now we can go back to the proof of Lemma \ref{square}. 
\begin{proof}
Recall that
\begin{align}
S_{4g,\text{e}}(V= \square) &= q^{2g+1}\sum_{\substack{f \in \mathcal{M}_{\leq 4g} \\ d(f) \text{ even}}} \frac{d_4(f)}{|f|^{\frac{3}{2}}}\sum_{\substack{C \in \mathcal{M}_{\leq y} \\ C | f^{\infty}}} \frac{1}{|C|^2} \bigg[ (q-1)  \sum_{ l \in \mathcal{M}_{\leq \frac{d(f)}{2}-g-2+d(C)} } G(l^2, \chi_f)  \nonumber \\
&- \sum_{l \in \mathcal{M}_{\frac{d(f)}{2}-g-1+d(C)}} G(l^2,\chi_f) - \frac{q-1}{q}  \sum_{ l \in \mathcal{M}_{\leq \frac{d(f)}{2}-g-1+d(C)} } G(l^2, \chi_f) + \frac{1}{q} \sum_{l \in \mathcal{M}_{\frac{d(f)}{2}-g+d(C)}} G(l^2,\chi_f) \bigg] .  \nonumber \end{align} 
 As for the main term, we evaluate the sum over $C$ as 
  \begin{equation}
   \sum_{\substack{ C \in \mathcal{M}_i \\ C | f^{\infty}}} \frac{1}{|C|^2} = \frac{1}{2\pi i}  \displaystyle \oint_{\Gamma_1} \frac{1}{q^{2i} u^{i+1} \prod_{P|f} (1-u^{d(P)})} \, du,
   \label{sumc2}
   \end{equation} where $\Gamma_1$ is a circle around the origin of radius less than $1$. In the expression for $S_{4g,\text{e}}(V = \square)$ above, let $n=d(f)$ and $i = d(C)$. Using Perron's formula twice, it follows that
   $$ S_{4g,\text{e}}(V=\square) = q^{2g+1} \frac{1}{ (2 \pi i)^3} \oint_{\Gamma_3} \oint_{\Gamma_2} \oint_{\Gamma_1} \sum_{\substack{n=0 \\ n \text{ even}}}^{4g} \frac{1}{q^n} \sum_{i=0}^y \frac{1}{q^{2i} u^{i}} \frac{(qx-1) \mathcal{A}(x,w,u)}{(1-x) x^{\frac{n}{2} - g+i} w^{n}} \left( 1- \frac{1}{qx} \right) \, \frac{du}{u} \, \frac{dw}{w} \, dx,$$ where $\Gamma_j$ is a circle around the origin of radius $R_j$ (for $j = 1,2,3$). Using the convergence of $\mathcal{A}(x,w,u)$ from Lemma \ref{auxiliary}, we initially pick $R_1= \frac{1}{q}, R_2= \frac{1}{q^2}, R_3= \frac{1}{q^{1+\epsilon}}.$ Using Lemma \ref{auxiliary} and computing the sums over $n$ and $i$ we get
   \begin{align}
 S_{4g,\text{e}}(V=\square) = -q^{2g+1} \frac{1}{ (2 \pi i)^3} \oint_{\Gamma_3} \oint_{\Gamma_2} \oint_{\Gamma_1}  & \frac{x^g (1-qxw)^4 \mathcal{B}(x,w,u)}{(1-qw)^4 (1-q^2w^2x)^{11} (q^2w^2x)^{2g} (1-x)(1-q^2ux)(q^2ux)^{y}}  \nonumber \\
 & \times \left(1-\frac{1}{qx} \right)  \, \frac{du}{u} \, \frac{dw}{w} \, dx \label{eq3}.
   \end{align}
   We enlarge the contour of integration $\Gamma_2$ to a circle around the origin $\Gamma_2'$ of radius $R_2'=q^{-\frac{1}{2} - \epsilon},$ and we encounter a pole at $w=q^{-1}$. Now we write
   \begin{align*}
&   \frac{1}{ 2 \pi i} \oint_{\Gamma_2}  \frac{x^g (1-qxw)^4 \mathcal{B}(x,w,u)}{(1-qw)^4 (1-q^2w^2x)^{11} (q^2w^2x)^{2g} (1-x)(1-q^2ux)(q^2ux)^{y}} \left(1- \frac{1}{qx} \right) \, \frac{dw}{w} = -\text{Res} (w =\frac{1}{q}) \\
   & +   \frac{1}{ 2 \pi i} \oint_{\Gamma_2'}  \frac{x^g (1-qxw)^4 \mathcal{B}(x,w,u)}{(1-qw)^4 (1-q^2w^2x)^{11} (q^2w^2x)^{2g} (1-x)(1-q^2ux)(q^2ux)^{y}} \left(1- \frac{1}{qx} \right) \, \frac{dw}{w} .
   \end{align*}
   Plugging this in equation \eqref{eq3}, we can bound the integral over the new contour by $q^{g(1+\epsilon)}$ (by keeping in mind the radii $R_1,R_2',R_3$). We compute the residue at $w=1/q$. Then
   \begin{align*}
    S_{4g,\text{e}}(V=\square) &= q^{2g+1} \frac{1}{(2 \pi i)^2} \oint_{\Gamma_3} \oint_{\Gamma_1} \frac{x^g \left( 1- \frac{1}{qx} \right) \mathcal{B}(x ,\frac{1}{q}, u) }{(1-q^2xu)(q^2xu)^y (1-x)^{11} x^{2g}} \Bigg[ 2 \Bigg( -32g^3(1-x)^3 \\
    &+48g^2(1-x)^2(-1+10x)-2g(1-x)(11-118x+1199x^2)-3+30x+63x^2+3990x^3 \Bigg) \\
    &+ (1-x) \Bigg( 6 ( 8g^2(1-x)^2 + 6g(1-x)(1-13x) +1-9x+190x^2) \frac{ \frac{1}{q} \frac{d}{dw} \mathcal{B}(x,w,u)}{\mathcal{B}(x,w,u)} \rvert_{w=\frac{1}{q}} \Bigg)  \\
    &+(1-x)^2 \Bigg( 3(-4g(1-x)-1+19x) \frac{ \frac{1}{q^2} \frac{d^2}{d w^2} \mathcal{B}(x,w,u)}{\mathcal{B}(x,w,u)} \rvert_{w=\frac{1}{q}} \Bigg) \\
    &+(1-x)^3   \frac{ \frac{1}{q^3} \frac{d^3}{d w^3} \mathcal{B}(x,w,u)}{\mathcal{B}(x,w,u)} \rvert_{w=\frac{1}{q}}   \Bigg] \, \frac{du}{u} \, dx.
   \end{align*}
    We shift the contour over $u$ to a circle $\Gamma_1'$ around the origin of radius $R_1'=\frac{1}{q^{\epsilon}},$ and we encounter a pole at $u = \frac{1}{q^2x}$. Let $f(x,u)$ be the integrand above. Then
   $$ S_{4g,\text{e}}(V= \square) = q^{2g+1} \frac{1}{2 \pi i} \oint_{\Gamma_3} - \text{Res} \left(u = \frac{1}{q^2x} \right) \, dx + q^{2g+1} \frac{1}{(2 \pi i)^2} \oint_{\Gamma_3} \oint_{\Gamma_1'} f(x,u) \, \frac{du}{u} \, dx.$$ 
   Note that in the double integral above, we can shift the contour of integration over $x$ to a circle $\Gamma_3'$ of radius $R_3'= q^{\epsilon/2},$ and we encounter a pole at $x=1$. Then it follows that the double integral is bounded by $ O(q^{2g-(2-\frac{\epsilon}{2}) \epsilon})$. Further computing the residue at $u=\frac{1}{q^2x}$ gives
   \begin{align*}
   S_{4g,\text{e}}(V= \square) &= q^{2g+1} \frac{1}{2 \pi i} \oint_{\Gamma_3} \frac{\mathcal{B}(x,\frac{1}{q},\frac{1}{q^2x}) \left(1-\frac{1}{qx} \right)} { (1-x)^{11} x^g}   \Bigg[ 2 \Bigg( -32g^3(1-x)^3 \\
    &+48g^2(1-x)^2(-1+10x)-2g(1-x)(11-118x+1199x^2)-3+30x+63x^2+3990x^3 \Bigg) \\
    &+ (1-x) \Bigg( 6 ( 8g^2(1-x)^2 + 6g(1-x)(1-13x) +1-9x+190x^2) \frac{ \frac{1}{q} \frac{d}{dw} \mathcal{B}(x,w,\frac{1}{q^2x})}{\mathcal{B}(x,w,\frac{1}{q^2x})} \rvert_{w=\frac{1}{q}} \Bigg)  \\
    &+(1-x)^2 \Bigg( 3(-4g(1-x)-1+19x) \frac{ \frac{1}{q^2} \frac{d^2}{d w^2} \mathcal{B}(x,w,\frac{1}{q^2x})}{\mathcal{B}(x,w,\frac{1}{q^2x})} \rvert_{w=\frac{1}{q}} \Bigg) \\
    &+(1-x)^3   \frac{ \frac{1}{q^3} \frac{d^3}{d w^3} \mathcal{B}(x,w,\frac{1}{q^2x})}{\mathcal{B}(x,w,\frac{1}{q^2x})} \rvert_{w=\frac{1}{q}}   \Bigg] \, dx+ O(q^{2g-(2-\frac{\epsilon}{2}) y}).
   \end{align*}
   In the integral above, there is a pole of order $11$ at $x=1$. Moreover, $\mathcal{B}(x,\frac{1}{q},\frac{1}{q^2x})$ is absolutely convergent for $|x|<\sqrt{q}$. Hence we shift the contour of integration to a circle of radius $q^{\frac{1}{2}-\epsilon}$ around the origin and compute the residue at $x=1$. The residue will give the main term of size $q^{2g+1}g^{10}$, and the integral around the circle of radius $q^{\frac{1}{2}-\epsilon}$ is bounded by $O(q^{\frac{3g}{2}(1+\epsilon)})$. This yields Lemma \ref{square}.
\end{proof}

\section{Bounding the contribution from non-square $V$}
Here we will bound the term $S(V \neq \square)$. Recall that $S_{4g}(V \neq \square) =S_{4g,\text{o}}+S_{4g,\text{e}}(V \neq \square)$, and $S_{4g,\text{o}}$ is given by equation ~\eqref{odd}. We'll prove the following.
\begin{lemma}
We have
$$ S_{4g}(V \neq \square) \ll  q^{2g+1} g^{9+\epsilon} q^{y \gamma} $$ and
$$S_{4g-1}(V \neq \square) \ll q^{2g+1} g^{9+\epsilon} q^{y \gamma},$$ where $\gamma$ is such that $1/\gamma^2 = o (\log g)$. 
\label{nons}
\end{lemma}
\begin{proof}
We will only bound the term $S_{4g,\text{o}}$, since the other terms are similar. In equation ~\eqref{odd}, write $S_{4g,\text{o}}=S_{1,\text{o}}-S_{2,\text{o}}$, with $S_{1,\text{o}}$ corresponding to the sum with $d(V) = d(f)-2g-2-2d(C).$  Write $V = V_0 V_1^2$, with $V_0$ a square-free monic polynomial. Let $r=d(V_0)$, with $r$ odd. Using equation \eqref{sumc} for the sum over $C$, we rewrite
\begin{align*}
S_{1,\text{o}} = q^{2g+1} \sqrt{q}  \frac{1}{2 \pi i}& \oint_{|u|=r_1} \sum_{\substack{n=0 \\ n \text{ odd}}}^{4g} \frac{1}{q^n} \sum_{i=0}^y \frac{1}{q^{2i} u^{i+1}} \sum_{\substack{r = 0 \\ r \text{ odd}}}^{n-2g-2+2i} \sum_{V_0 \in \mathcal{H}_r} \sum_{V_1 \in \mathcal{M}_{\frac{n-2g-2+2i-r}{2}}} \\
& \times  \sum_{f \in \mathcal{M}_n} \frac{d_4(f) G(V_0 V_1^2,\chi_f)}{\sqrt{|f|} \prod_{P|f} (1-u^{d(P)})} \, du,
\end{align*} where $r_1<1$.
Let $$ \mathcal{F} (V_0 ; w,x,u) := \sum_{V_1 \in \mathcal{M}} \sum_{f \in \mathcal{M}} \frac{d_4(f) G(V_0V_1^2, \chi_f) }{\sqrt{|f|} \prod_{P|f} (1-u^{d(P)})} w^{d(f)} x^{d(V_1)}.$$

Then 
\begin{align*}
& \mathcal{F} (V_0 ; w,x,u) = \mathcal{Z}(x)  \prod_{P \nmid V_0} \bigg( 1+ \frac{4 \left( \frac{V_0}{P} \right) w^{d(P)} (1+|P|w^{2d(P)}x^{d(P)})(1-x^{d(P)})}{(1-u^{d(P)})(1-|P|w^{2d(P)}x^{d(P)})^4} \bigg)\\
  &+  \frac{(|P|-1)w^{2d(P)}x^{d(P)} (10-5|P|w^{2d(P)}x^{d(P)}+4|P|^2w^{4d(P)}x^{2d(P)}-|P|^3w^{6d(P)}x^{3d(P)})}{(1-u^{d(P)})(1-|P|w^{2d(P)}x^{d(P)})^4} \\
  &\prod_{P|V_0} \bigg( 1- \frac{w^{2d(P)}(1-|P|x^{d(P)})(10-5|P|w^{2d(P)}x^{d(P)}+4|P|^2w^{4d(P)}x^{2d(P)}-|P|^3w^{6d(P)}x^{3d(P)})}{(1-u^{d(P)})(1-|P|w^{2d(P)}x^{d(P)})^4} \bigg).
  \end{align*} 
  Note that we can further write
  $$ \mathcal{F} (V_0 ; w,x,u)  = \mathcal{Z}(x) \mathcal{Z}(qw^2x)^{10} \mathcal{L}(w,\chi_{V_0})^4 \mathcal{L}(wu,\chi_{V_0})^4 \mathcal{L}(wu^2 ,\chi_{V_0})^4 \cdot \ldots.$$
  Using Perron's formula twice, we have that
   \begin{align}
  S_{1,\text{o}} &= q^{2g+1}  \sqrt{q}  \frac{1}{(2 \pi i)^3} \oint \oint \oint x^g \sum_{\substack{n=0 \\ n \text{ odd}}}^{4g} \frac{1}{q^n w^{n+1} x^{n/2}} \sum_{i=0}^{y} \frac{1}{q^{2i}u^{i+1} x^i} \sum_{\substack{r=0 \\ r \text{ odd}}}^{n-2g-2+2i} x^{r/2} \nonumber \\
  & \times \sum_{V_0 \in \mathcal{H}_r} \mathcal{F}(V_0;w,x,u) \, du \,dw \, dx, \label{sodd}
  \end{align} where $|x|<1/q, |u|<1$ and $ |q^2w^2x|<1$. We pick $|x|=1/q^{1+\epsilon}, |w|=1/q^{1/2}$ and $|u|=1/q^{\gamma}$, where $\gamma \to 0$ as $g \to \infty$. Let $l$ be an integer such that $|wu^{l-1} | \geq 1/q$ and $|wu^l|<1/q$.
  We will prove the following.
  \begin{lemma}
  Let $|w|=q^{-\frac{1}{2}}, |u|=q^{-\gamma}$, with $\gamma$ depending on $g$ such that $ 1/\gamma^2 = o(\log g)$. Let $l$ be an integer such that $|wu^{l-1}| \geq q^{-1}$ and $|wu^l|<q^{-1}$. Let $k$ be a positive real number and $\epsilon>0$. Then
  $$ \sum_{D \in \mathcal{H}_{2g+1}}  \left| \mathcal{L}(w,\chi_D) \mathcal{L}(wu,\chi_D) \cdot \ldots \cdot \mathcal{L}(wu^{l-1},\chi_D) \right|^k \ll q^{2g+1} g^{\epsilon} \exp \left( k \mathcal{M}(w,g)+ \frac{k^2}{2} \mathcal{V}(w,g) \right),$$ with
  $\mathcal{M}(w,g)$ and $\mathcal{V}(w,g)$ as in Theorem ~\ref{upperbound}.
  \label{shiftedupperb}
  \end{lemma}
  We postpone the proof of the Lemma to section ~\ref{upper}.
  
  Now we use the bound above in equation ~\eqref{sodd} as follows. We write the integral over $w$ as $\oint_{|w|=1/\sqrt{q}} = \int_{C_1} + \int_{C_2}$, where $C_1$ is the arc around $q^{-\frac{1}{2}}$ of angle $1/2r$, and $C_2$ is its complement. Using Lemma ~\ref{shiftedupperb} for $k=4$, the expressions for $\mathcal{M}(w,g)$ and $\mathcal{V}(w,g)$ and since the arc $C_1$ has length of size $1/r$, we have that $$ \left|\int_{C_1} \frac{1}{w^{n+1}} \sum_{V_0 \in \mathcal{H}_r} \mathcal{F}(V_0;w,x,u) \, dw \right| \ll q^{n/2+r} r^{10+\epsilon} \frac{1}{r} \ll q^
  {n/2+r} r^{9+\epsilon}.$$ For the integral over $C_2$, we make a change of variables $w=\frac{1}{\sqrt{q}} e^{i \theta}$, and then 
  $$ \left|\int_{C_1} \frac{1}{w^{n+1}} \sum_{V_0 \in \mathcal{H}_r} \mathcal{F}(V_0;w,x,u) \, dw \right| \ll q^{n/2+r} r^{4+\epsilon} \int_{1/2r}^{2 \pi - 1/2r} \frac{1}{ \theta^6} \, d\theta \ll q^
  {n/2+r} r^{9+\epsilon} .$$ 
  Then
  $$ \left|\frac{1}{2 \pi i} \oint_{|w|=q^{-\frac{1}{2}}}  \frac{1}{w^{n+1}} \sum_{V_0 \in \mathcal{H}_r} \mathcal{F}(V_0;w,x,u) \, dw \right| \ll q^
  {n/2+r} r^{9+\epsilon}.$$
  Plugging this in ~\eqref{sodd} and trivially bounding everything else, we get that $$S_{1,\text{o}} \ll q^{2g+1} g^{9+\epsilon} q^{y \gamma},$$ which finishes the proof of Lemma ~\ref{nons}. 
  
\end{proof}
\section{Proof of Theorem \ref{fourth}}
In this section, we obtain the asymptotic formula in Theorem \eqref{fourth}. We choose $y = 100 \log g$, and $\gamma=  (\log g)^{(\epsilon-1)/2}$. Using Lemma \ref{tailc}, equation \eqref{maint} and Lemma \ref{square}, we get that
$$ \sum_{D \in \mathcal{H}_{2g+1}} L \left( \frac{1}{2} ,\chi_D \right)^4 = q^{2g+1} g^{10} a_{10} + O(q^{2g+1} g^{9+\epsilon}),$$ where $a_{10}$ is a coefficient which can be computed explicitly (see equation \eqref{coeff1}).

 Now let $\alpha<4g$ be an even integer which we will choose later. For now, we think of $\alpha$ as being on the scale of $\log g$. We write
\begin{align}
\sum_{D \in \mathcal{H}_{2g+1}} L \left( \frac{1}{2},\chi_D \right)^4 &= \sum_{D \in \mathcal{H}_{2g+1}} \sum_{f \in \mathcal{M}_{\leq 4g-\alpha}} \frac{d_4(f) \chi_D(f)}{\sqrt{|f|}} + \sum_{D \in \mathcal{H}_{2g+1}} \sum_{\substack{f \in \mathcal{M} \\ 4g-\alpha<d(f) \leq 4g}} \frac{d_4(f) \chi_D(f)}{\sqrt{|f|}} \nonumber \\
& +  \sum_{D \in \mathcal{H}_{2g+1}} \sum_{f \in \mathcal{M}_{\leq 4g-1-\alpha}} \frac{d_4(f) \chi_D(f)}{\sqrt{|f|}} + \sum_{D \in \mathcal{H}_{2g+1}} \sum_{\substack{f \in \mathcal{M} \\ 4g-1-\alpha<d(f) \leq 4g-1}} \frac{d_4(f) \chi_D(f)}{\sqrt{|f|}} . \label{small}
\end{align}
Denote the first term above by $M_{1,\alpha}$, the second by $T_1,\alpha$, the third by $M_{2,\alpha}$ and the fourth by $T_{2,\alpha}.$ We will focus on $M_{1,\alpha}$ and $T_{1,\alpha}$, since the other two terms are similar.

\subsection{The term $M_{1,\alpha}$}

We treat the term $M_{1,\alpha}$ in the same way as we treated the full sum over polynomials $f$ in the previous sections. Write $M_{1,\alpha}= M_{1,\alpha}(V=0)+M_{1,\alpha}(V= \square)+M_{1,\alpha}(V \neq \square)$. Using the same methods as before, $M_{1,\alpha}(V \neq \square) \ll q^{2g- \frac{\alpha}{2}} g^{9+\epsilon}$, so $M_{1,\alpha}(V \neq \square) = o(q^{2g+1}).$
We have the following.
\begin{lemma} 
\label{maintrunc}
Keeping the previous notation, we have
\begin{align*}
M_{1,\alpha}(V = 0) &= \frac{q^{2g+1}}{10! \zeta_q(2)} \Bigg( (2g)^{10} Q_0(\alpha) + (2g)^9 Q_1(\alpha) +(2g)^8 Q_2(\alpha) \Bigg) +O\Big(q^{2g+1} g^{7+\epsilon}\Big),
\end{align*}
where $Q_i(\alpha)$ is a polynomial of degree $10-i$, for $i=0,1,2$. Moreover, $Q_i(x)$ can be written down explicitly (see formulas \eqref{q0}, \eqref{q1} and \eqref{q2}.)
 Also
 \begin{align*}
 M_{1,\alpha}(V= \square) = - \frac{q^{2g+1}}{10! \zeta_q(2)} \Bigg( g^{10} R_0(\alpha) + g^9 R_1(\alpha) +g^8 R_2(\alpha) \Bigg) + O\Big(q^{2g+1} g^{7+\epsilon}\Big),
 \end{align*}
 where $R_i(\alpha)$ is a polynomial of degree $10-i$ (see formulas \eqref{r0}, \eqref{r1} and \eqref{r2}.)
\end{lemma}
\begin{proof}
Similarly as for the term $M_{4g}$, we have
$$ M_{1,\alpha}(V=0) =  \frac{q^{2g+1}}{\zeta_q(2)} \frac{1}{(2 \pi i)^2} \oint_{|w|=\frac{1}{q^{1+\epsilon}}} \oint_{|u|=\frac{1}{q^{2+\epsilon}}} \frac{\mathcal{H}(w,u)}{(1-qw)^{11}(qw)^{2g-\frac{\alpha}{2}} (1-q^2u)(q^2u)^y} \, \frac{du}{u} \, \frac{dw}{w},$$
with $\mathcal{H}(w,u)$ as in \eqref{hw}. For simplicity, since we will only evaluate $\mathcal{H}(w,u)$ at $u=\frac{1}{q^2}$, we let $\mathcal{H}(w,\frac{1}{q^2}) = \mathcal{H}(w)$. We proceed similarly as in section \ref{maintermsec}, and then the expression for $M_{1,\alpha}(V=0)$ follows by a residue calculation. Note that we can write down lower order terms for $M_{1,\alpha}(V=0)$, but for our purposes it is enough to consider the coefficients down to $g^7$. By the residue computation, we have 
\begin{align}
& Q_0(\alpha) = \bq \label{q0} \\
&Q_1(\alpha) = - 5 \alpha \bq + 55 \bq- 10 \bprime \label{q1} \\
& Q_2(\alpha) =  \frac{ 45 \alpha^2 }{4} \bq+ \alpha \left( -\frac{495 \bq}{2} + 45 \bprime \right) +1320\bq-  450\bprime +  45 \bdoubleprime \label{q2}.
\end{align}

As in section \ref{secmain}, we also have
\begin{align*}
M_{1,\alpha}(V=\square) = -q^{2g+1} \frac{1}{ (2 \pi i)^3} \oint_{\Gamma_3} \oint_{\Gamma_2} \oint_{\Gamma_1}  & \frac{x^g (1-qxw)^4 \mathcal{B}(x,w,u)}{(1-qw)^4 (1-q^2w^2x)^{11} (q^2w^2x)^{2g-\frac{\alpha}{2}} (1-x)(1-q^2ux)(q^2ux)^{y}}   \\
 & \times \left(1-\frac{1}{qx} \right)  \, \frac{du}{u} \, \frac{dw}{w} \, dx ,
   \end{align*}
with $\Gamma_i$ as for $S_{4g,\text{e}}(V=\square)$. Since we will only evaluate $\mathcal{B}(x,w,u)$ at $u= \frac{1}{q^2 x}$, for simplicity, we let $\mathcal{C}(x,w) = \mathcal{B}(x,w,\frac{1}{q^2x}).$ The expression for $M_{1,\alpha}(V=\square)$ follows by a direct residue computation. We have
\begin{align}
& R_0(\alpha) = 640 \cq \label{r0} \\
& R_1(\alpha) = -2100 \alpha \cq - 2000 \zeta_q(2) \cq+20100 \cq -1100 \cprimew - 2000 \cprimex \label{r1} \\
& R_2(\alpha) = 2880 \alpha^2 \cq + \alpha \Big( 4140 \zeta_q(2) \cq - 57150 \cq +3690 \cprimew +4140 \cprimex \Big) \nonumber \\
&  \qquad  \quad -48420 \zeta_q(2) \cq +269430 \cq - 32670 \cprimew+4860 \zeta_q(2) \cprimew  \nonumber \\
&  \qquad  \quad +  630 \frac{ \frac{d^2}{dw^2} \mathcal{C} (1,w) \rvert_{w=\frac{1}{q}}}{q^2} -1440 \zeta_q(2) \cprimex -48420 \cprimex + 4860  \frac{ \frac{d}{dw} \frac{d}{dx}  \mathcal{C} (x,w) \rvert_{x=1,w=\frac{1}{q}}}{q} \nonumber \\
& \qquad  \quad - 720  \frac{d^2}{dx^2} \mathcal{C} (x, \frac{1}{q}) \rvert_{x=1}. \label{r2}
\end{align}
\end{proof}
\begin{remark}
Note that the sum $M_{1,\alpha}(V=0)+ M_{1,\alpha}(V = \square)$ only involves terms of size $g^9 \alpha$, and $g^8 \alpha$ (the $g^8 \alpha^2$ term in $M_{1,\alpha}(V=0)$ cancels out with the corresponding term in $M_{1,\alpha}(V=\square)$). We will show that these two terms cancel out the contribution from $T_{1,\alpha}$.
\end{remark}

\subsection{The term $T_{1,\alpha}$}

Using Perron's formula, we write
\begin{equation}
T_{1,\alpha} =  \frac{1}{2 \pi i}  \oint_{C} \frac{1 - z^{\alpha}}{(1-z)z^{4g+1}} \sum_{D \in \mathcal{H}_{2g+1}} \mathcal{L} \left(  \frac{z}{\sqrt{q}} , \chi_D \right)^4 \, du,
\label{talpha}
\end{equation}
where we pick $C$ to be a circle of radius $1$ around the origin. We further decompose the circle into two arcs: an arc $C_1$ of angle $\theta_0$ around $1$, and its complement $C_2$. We will choose $\theta_0$ later, but for now we keep in mind that $\theta_0$ is a function of $g$, $\theta_0 \to 0 $ as $g \to \infty$ and $g \theta_0  \to \infty$. Then $T_{1,\alpha} = T_{1, C_1} + T_{1,C_2}$, where
\begin{equation}
T_{1,C_1} = \frac{1}{2 \pi i} \int_{C_1} \frac{1 - z^{\alpha}}{(1-z)z^{4g+1}} \sum_{D \in \mathcal{H}_{2g+1}} \mathcal{L} \left(  \frac{z}{\sqrt{q}} , \chi_D \right)^4 \, du ,
\label{tt1}
\end{equation}
and $T_{1,C_2}$ is similarly defined. We treat $T_{1,C_1}$ and $T_{1,C_2}$ differently. For $T_{1,C_2}$, we will use upper bounds for moments of $L$--functions. On $C_1$, we will use an asymptotic formula for the shifted moment $\sum_{D} \mathcal{L}(u/\sqrt{q},\chi_D)^4$ with an error of size $g^{9+\epsilon}.$ Since $C_1$ is a small arc around $1$, we will manage to get a better error bound. Then we will explicitly compute terms of size $g^9 \alpha, g^8 \alpha^2$ and $g^8 \alpha$ which will cancel out the contribution from $M_{1,\alpha}$.

We will prove the following.
\begin{lemma}
With the same notation as before, 
$$T_{1,C_2} \ll q^{2g+1} g^{4+\epsilon} \theta_0^{-5}.$$ \label{t2}
\end{lemma}
\begin{proof}
The bound easily follows from Corollary \ref{cor4} by integrating $\theta^{-6}$ along the arc $C_2$.
\end{proof}

Now we focus on $T_{1,C_1}$. Similarly as in the previous sections, we can obtain an asymptotic formula for the shifted moment $\sum_D \mathcal{L}(z/\sqrt{q},\chi_D)^4$ when $z$ is on the arc $C_1$ (so close to $1$). By adapting the proof for finding an asymptotic for the fourth moment at the critical point, we have the following
\begin{align}
& \sum_{D \in \mathcal{H}_{2g+1}}  \mathcal{L} \left( \frac{z}{\sqrt{q}} , \chi_D \right)^4  = \frac{q^{2g+1}}{\zeta_q(2)} \frac{1}{ (2 \pi i)^2 } \oint \oint \frac{\mathcal{H}(w,u)}{(1-qw)^{10} (1-\frac{qw}{z^2}) ( \frac{qw}{z^2} )^{2g} (1-q^2u)(q^2u)^y} \, \frac{du}{u} \, \frac{dw}{w}  \label{firstt}  \\
&+  \frac{q^{2g+1}}{\zeta_q(2)} \frac{1}{ (2 \pi i)^2} \oint \oint \frac{ z^{4g+2} \mathcal{H}(w,u)}{(1-qw)^{10} (1-qwz^2) ( qw)^{2g-1}(1-q^2u)(q^2u)^y}  \, \frac{du}{u} \, \frac{dw}{w}  \label{secondt} \\
&-  \frac{q^{2g+1}}{(2 \pi i)^3} \oint \oint \oint \frac{x^g (1-qwx)^4 \mathcal{B}(x,w,u)(1-\frac{1}{qx})}{(1-x)(1-qw)^4(1-q^2w^2x)^{10} (1- \frac{q^2w^2x}{z^2}) ( \frac{q^2w^2x}{z^2})^{2g}(1-q^2ux)(q^2ux)^y} \, \frac{du}{u} \, \frac{dw}{w} \, dx  \label{thirdt}  \\
&-  \frac{q^{2g+1}}{(2 \pi i)^3} \oint \oint \oint \frac{z^{4g+2} x^g (1-qwx)^4 \mathcal{B}(x,w,u)(1-\frac{1}{qx})}{(1-x)(1-qw)^4(1-q^2w^2x)^{10}(1- q^2w^2xz^2) ( q^2w^2x)^{2g-1}(1-q^2ux)(q^2ux)^y} \, \frac{du}{u}  \, \frac{dw}{w} \, dx  \label{fourtht} \\
&+ O(q^{2g+1} g^{9+\epsilon}), \label{errorshift}
\end{align} where the contours of integration are the same as in sections \ref{maintermsec} and \ref{secmain}.
Now we plug this in the integral \eqref{tt1} for $T_{1,C_1}$. When we integrate the shifted moment along the arc $C_1$, we obtain two terms: one corresponding to the main term above (coming from $V=0$), which we denote by $T_{1,C_1}(V=0)$ (this term is the sum of the integrals over $C_1$ of \eqref{firstt} and \eqref{secondt}), and another term corresponding to the secondary main term in the shifted moment. We denote this term by $T_{1,C_1}(V= \square)$, and it is the sum of the integrals over $C_1$ of \eqref{thirdt} and \eqref{fourtht}. The error term \eqref{errorshift} will become $O(q^{2g+1} g^{9+\epsilon} \theta_0)$ in the integral for $T_{1,C_1}$. Using this bound and Lemma \ref{t2}, it follows that
\begin{align}
\sum_{D \in \mathcal{H}_{2g+1}} L \left( \frac{1}{2} , \chi_D \right)^4 &= M_{1,\alpha}(V=0)+M_{1,\alpha}(V=\square)+M_{2,\alpha}(V=0)+M_{2,\alpha}(V=\square)+ o(q^{2g+1}) \\
&+ T_{1,C_1}(V=0)+T_{1,C_1}(V=\square)+T_{2,C_1}(V=0)+T_{2,C_1}(V=\square) \label{new} \\
&+O\Big(q^{2g+1} g^{9+\epsilon} \theta_0+ q^{2g+1} g^{4+\epsilon} \theta_0^{-5} \Big). \nonumber 
\end{align}
We rewrite
\begin{align}
T_{1,C_1}(V=0) &= \frac{q^{2g+1}}{\zeta_q(2)} \frac{1}{(2 \pi i)^3} \int_{C_1} \oint \oint \frac{1- z^{\alpha}}{(1-z) z^{4g+1}} \Bigg( \frac{\mathcal{H}(w,u)}{(1-qw)^{10} (1-\frac{qw}{z^2}) ( \frac{qw}{z^2} )^{2g}(1-q^2u)(q^2u)^y} \nonumber \\
&+ \frac{ z^{4g+2} \mathcal{H}(w,u)}{(1-qw)^{10} (1-qwz^2) ( qw)^{2g-1}(1-q^2u)(q^2u)^y}  \Bigg) \,\frac{du}{u} \, \frac{dw}{w} \, dz ,\label{tmain}
\end{align}
and
\begin{align}
T_{1,C_1}(V= \square) &= - \frac{q^{2g+1} }{(2 \pi i)^4} \int_{C_1} \oint \oint \oint \frac{1- z^{\alpha}}{(1-z) z^{4g+1}} \\
& \times  \Bigg( \frac{x^g (1-qwx)^4 \mathcal{B}(x,w,u)(1-\frac{1}{qx})}{(1-x)(1-qw)^4(1-q^2w^2x)^{10} (1- \frac{q^2w^2x}{z^2}) ( \frac{q^2w^2x}{z^2})^{2g} (1-q^2ux)(q^2ux)^y}  \nonumber  \\
& + \frac{z^{4g+2} x^g (1-qwz)^4 \mathcal{B}(x,w,u)(1-\frac{1}{qx})}{(1-x)(1-qw)^4(1-q^2w^2x)^{10}(1- q^2w^2xz^2) ( q^2w^2x)^{2g-1}(1-q^2ux)(q^2ux)^y} \Bigg) \, \frac{du}{u} \, \frac{dw}{w} \, dx \, dz  . \label{tsquare}
\end{align}
In the next lemma, we obtain asymptotic formulas for $T_{1,C_1}(V=0)$ and $T_{1,C_1}(V=\square)$.
\begin{lemma} \label{compt1}
Keeping the previous notation, we have the following.
\begin{align*}
T_{1,C_1}(V=0) &= \frac{q^{2g+1}}{9! \zeta_q(2)} \cdot \frac{\alpha}{2} \Bigg( c_9 (2g)^ 9   -  (2g)^8 \bq \binom{9}{1}  \Big( \frac{\cos(2 \pi \theta_0)}{2 \pi \theta_0} + \sum_{m=1}^{2g} m \sin( 4 \pi m \theta_0) (A(m)+B(m)) \Big) \\
& + c_8 (2g)^8   \Bigg) +O \Big(q^{2g+1} g^8 \theta_0 \alpha^3 +q^{2g+1}g^7 \theta^{-1} \alpha \Big),
\end{align*} with $c_9$ and $c_8$ given by \eqref{c9} and \eqref{c8} respectively, and
 $A(m),B(m)$ can be written down explicitly and are such that $A(m),B(m) = O (1/m^2)$ (see formula \eqref{aterm}).
Also 
\begin{align*}
T_{1,C_1}(V=\square) &= \frac{q^{2g+1}}{9! \zeta_q(2)} \cdot \frac{\alpha}{2}  \Bigg( f_9 g^9  + (2g)^8 \cq \binom{9}{1}  \Big( \frac{\cos(2 \pi \theta_0)}{2 \pi \theta_0}  + \sum_{m=1}^{g} m \sin( 4 \pi m \theta_0) (A(m)+B(m)) \Big) \\
& + f_8 g^8  \Bigg)  +O\Big(q^{2g+1} g^8 \theta_0 \alpha^3 +q^{2g+1}g^7 \theta_0^{-1} \alpha \Big),
\end{align*}
with $f_9$ and $f_8$ as in \eqref{f9} and \eqref{f8}, and $A(m)$ and $B(m)$ as before. 
\end{lemma}

\begin{proof}
Recall the expression \eqref{tmain} for $T_{1,C_1}(V=0)$, which is a sum of two terms. We rewrite
\begin{align*}
T_{1,C_1}(V=0) &= \frac{q^{2g+1}}{\zeta_q(2)} \frac{1}{(2 \pi i)^2} \oint \oint \frac{ \mathcal{H}(w,u)}{(1-qw)^{10} (qw)^{2g}(1-q^2u)(q^2u)^y} \Bigg[ \frac{1}{2 \pi i} \int_{C_1} \frac{1-z^{\alpha}}{z(1-z)(1-\frac{qw}{z^2})} \, dz + \\
& +  \frac{1}{2 \pi i} \int_{C_1} \frac{(1-z^{\alpha})zqw}{(1-z)(1-qwz^2)} \, dz \Bigg]  \, \frac{du}{u} \,  \frac{dw}{w} \label{t1}
\end{align*}
Now 
\begin{align*}
 \frac{1}{2 \pi i} \int_{C_1} \frac{1-z^{\alpha}}{z(1-z)(1-\frac{qw}{z^2})} \, dz &= \sum_{j=0}^{\alpha-1} \sum_{m=0}^{\infty} (qw)^m \frac{1}{2 \pi i} \int_{C_1} z^{j-2m-1} \, dz = \sum_{j=0}^{\alpha-1} \sum_{m=0}^{\infty} (qw)^m  \int_{- \theta_0}^{\theta_0} e^{2 \pi i \theta(j-2m)} \, d \theta\\
 & = \frac{1}{ \pi } \sum_{j=0}^{\alpha-1} \sum_{m=0}^{\infty} (qw)^m  \frac{ \sin (2 \pi \theta_0(2m-j))}{2m-j}.
\end{align*}
Similarly
\begin{align*}
   \frac{1}{2 \pi i} \int_{C_1} \frac{(1-z^{\alpha})zqw}{(1-z)(1-qwz^2)} \, dz &= \frac{1}{ \pi } \sum_{j=0}^{\alpha-1} \sum_{m=0}^{\infty} (qw)^{m+1}  \frac{ \sin (2 \pi \theta_0(2m+j+2))}{2m+j+2} \\
   &= \frac{1}{ \pi } \sum_{j=0}^{\alpha-1} \sum_{m=1}^{\infty} (qw)^{m}  \frac{ \sin (2 \pi \theta_0(2m+j))}{2m+j}   .
   \end{align*}
   Plugging these in the expression for $T_{1,C_1}(V=0)$, we get that
   \begin{align*}
   T_{1,C_1}(V=0)  &= \frac{q^{2g+1}}{\pi \zeta_q(2)} \frac{1}{(2 \pi i)^2}\oint \oint  \frac{ \mathcal{H}(w,u)}{(1-qw)^{10} (qw)^{2g}(1-q^2u)(q^2u)^y} \Bigg[ \sum_{j=0}^{\alpha-1} \sum_{m=0}^{2g} (qw)^m \frac{ \sin (2 \pi \theta_0(2m-j))}{2m-j} \\
   &+  \sum_{j=0}^{\alpha-1} \sum_{m=1}^{2g} (qw)^{m}  \frac{ \sin (2 \pi \theta_0(2m+j))}{2m+j} \Bigg] \, \frac{du}{u} \, \frac{dw}{w}.
   \end{align*}
   Note that in the expression above, there is a pole of order $10$ at $w=\frac{1}{q}$ and a simple pole at $u=\frac{1}{q^2}$. Similarly as before, we evaluate the residue at $w=\frac{1}{q}$ and $u=\frac{1}{q^2}$, and we get that
   \begin{align*}
   T_{1,C_1}(V=0) &= \frac{q^{2g+1}}{9! \pi \zeta_q(2) } \Bigg[ \sum_{j=0}^{\alpha-1} \sum_{m=0}^{2g}\frac{ \sin (2 \pi \theta_0(2m-j))}{2m-j} Q(2g-m) +  \sum_{j=0}^{\alpha-1} \sum_{m=1}^{2g}  \frac{ \sin (2 \pi \theta_0(2m+j))}{2m+j} Q(2g-m) \Bigg] \\
   &+O(q^{2g+1}\alpha),
   \end{align*}
 where $Q(x) = \sum_{i=0}^9 c_i x^i$ is a polynomial of degree $9$ whose coefficients can be computed explicitly. For example,
 \begin{equation}
 c_9 = \bq,
 \label{c9}
 \end{equation}
 \begin{equation}
 c_8 = 45 \bq - 9  \frac{\bprime}{q}. \label{c8} 
 \end{equation}
 Now $Q(2g-m)$ is a polynomial in $2g$ and $m$ of total degree less than or equal to $9$. We write
 \begin{align}
  T_{1,C_1}(V=0) &= \frac{q^{2g+1}}{9! \pi \zeta_q(2) } \sum_{r=0}^9 c_r \Bigg[ \sum_{j=0}^{\alpha-1} \sum_{m=0}^{2g}\frac{ \sin (2 \pi \theta_0(2m-j))}{2m-j} (2g-m)^r +  \sum_{j=0}^{\alpha-1} \sum_{m=1}^{2g}  \frac{ \sin (2 \pi \theta_0(2m+j))}{2m+j} (2g-m)^r \Bigg] \\
  &+O(q^{2g+1}\alpha).
 \end{align}
 By Lemma \ref{ak}, the term with $r \leq 7$ will be of size $q^{2g+1}g^7 \alpha$. Then we only need to consider $r =8$ and $r=9$. When $r=8$, write $(2g-m)^ 8 = \sum_{k=0}^8 (-1)^k \binom{8}{k} m^k (2g)^{8-k}$. If $k \geq 1$, by Lemma \ref{ak}, this term will be bounded by $q^{2g+1}g^7 \theta_0^{-1} \alpha$. So when $r=8$, we only consider $k=0$. Using Lemma \ref{ak} again, we have
 \begin{align*}
 T_{1,C_1}(V=0) &= \frac{q^{2g+1}}{9!\pi \zeta_q(2) } c_8 (2g)^8 \frac{\pi \alpha}{2} + \frac{q^{2g+1}}{9!\pi \zeta_q(2) }   c_9 \Bigg[ \sum_{j=0}^{\alpha-1} \sum_{m=0}^{2g}\frac{ \sin (2 \pi \theta_0(2m-j))}{2m-j} (2g-m)^9 \\
 &+  \sum_{j=0}^{\alpha-1} \sum_{m=1}^{2g}  \frac{ \sin (2 \pi \theta_0(2m+j))}{2m+j} (2g-m)^9 \Bigg] +O(q^{2g+1}g^7 \theta_0^{-1} \alpha).
 \end{align*}
 Write $(2g-m)^9 = \sum_{k=0}^9 (-1)^k \binom{9}{k} m^k (2g)^{9-k} $. We have to evaluate sums of the form
 $$  A(k, \theta_0)= \sum_{j=0}^{\alpha-1} \sum_{m=0}^{2g}  m^k  \frac{ \sin (2 \pi \theta_0(2m-j))}{2m-j}+\sum_{j=0}^{\alpha-1} \sum_{m=1}^{2g}  m^k  \frac{ \sin (2 \pi \theta_0(2m+j))}{2m+j},$$ for $ k \leq 9$. Then
 $$ T_{1,C_1}(V=0) = \frac{q^{2g+1}}{9! \zeta_q(2) } c_8 (2g)^8 \frac{\alpha}{2} + \frac{q^{2g+1}}{9!\pi \zeta_q(2) }   c_9 \sum_{k=0}^9 (-1)^k \binom{9}{k} (2g)^{9-k}  A(k,\theta_0)+ O(q^{2g+1}g^7 \theta_0^{-1} \alpha).$$
 Using Lemma \ref{ak}, we have
 \begin{align*}
  T_{1,C_1}(V=0) &= \frac{q^{2g+1}}{ 9!\zeta_q(2) } c_8 (2g)^8 \frac{ \alpha}{2} +  \frac{q^{2g+1}}{9!\zeta_q(2) }   c_9 (2g)^9 \frac{ \alpha}{2} - \frac{9q^{2g+1}}{ 9! \pi \zeta_q(2) } c_9  (2g)^8 \Bigg( \frac{ \alpha \cos(2 \pi \theta_0)}{2 \sin(2\pi \theta_0)} \\
  &+ \sum_{m=1}^{2g} m \sin(4 \pi m \theta_0)(A(m)+B(m)) \Bigg) -  \frac{q^{2g+1}}{ 9!  \pi \zeta_q(2) }  c_9 \frac{\alpha}{2} \sum_{k=0}^9 (2g)^{9-k} (-1)^k \binom{9}{k} \\
  &\Bigg( (2g)^{k -1} \frac{ \cos(8g \pi \theta_0)}{2 \pi \theta_0} - (2g)^{k-1} \sin(8g \pi \theta_0) \Bigg) + O\Big(q^{2g+1}(g^7 \theta_0^{-1} \alpha + g^8 \theta_0 \alpha^3) \Big).
   \end{align*}
   Using the fact that $\sum_{k=0}^9 (-1)^k \binom{9}{k}=0$, the formula for $T_{1,C_1}(V=0)$ in Lemma \ref{compt1} follows. 

Now we focus on $T_{1,C_1}(V=\square)$. We will skip some of the details, since they are similar to the ideas used when dealing with $T_{1,C_1}(V=0)$.  We have
\begin{align*}
T_{1,C_1}(V=\square) &= -\frac{q^{2g+1}}{(2 \pi i)^3} \oint \oint \oint \frac{x^g (1-qwx)^4 \mathcal{B}(x,w,u)(1-\frac{1}{qx})}{(1-x)(1-qw)^4(1-q^2w^2x)^{10} (q^2w^2x)^{2g}(1-q^2ux)(q^2ux)^y} \\
& \times \Bigg[ \sum_{j=0}^{\alpha-1} \sum_{m=0}^{g} (qw)^m \frac{ \sin (2 \pi \theta_0(2m-j))}{2m-j} +  \sum_{j=0}^{\alpha-1} \sum_{m=1}^{g} (qw)^{m}  \frac{ \sin (2 \pi \theta_0(2m+j))}{2m+j} \Bigg] \, \frac{du}{u} \, \frac{dw}{w} \, dx.
\end{align*}
Similarly as in section \ref{secmain}, we shift contours, and evaluate the pole at $w=1/q$. Then we encounter a simple pole at $u=\frac{1}{q^2x}$ and a pole of order $10$ at $x=1$. Computing the residue, we get that
\begin{align}
T_{1,C_1}(V=\square) &= -q^{2g+1} \frac{1}{3! 9!} \sum_{\substack{i+k \leq 9 \\ i \leq 3}} c_{ik}   \Bigg[ \sum_{j=0}^{\alpha-1} \sum_{m=0}^{g} \frac{ \sin (2 \pi \theta_0(2m-j))}{2m-j}(g-m)^k (2g-m)^i  \nonumber \\
&+  \sum_{j=0}^{\alpha-1} \sum_{m=1}^{g} \frac{ \sin (2 \pi \theta_0(2m+j))}{2m+j} (g-m)^k (2g-m)^i \Bigg] +O(q^{2g+1}\alpha)  \label{firstexpr} \\
&= -\frac{q^{2g+1}}{\zeta_q(2) 9!} \sum_{n+r \leq 9}  d_{nr} g^n    \Bigg[ \sum_{j=0}^{\alpha-1} \sum_{m=0}^{g} \frac{ \sin (2 \pi \theta_0(2m-j))}{2m-j} m^r +  \sum_{j=0}^{\alpha-1} \sum_{m=1}^{g} \frac{ \sin (2 \pi \theta_0(2m+j))}{2m+j} m^r \Bigg]  \label{secondexpr} \\
&+O(q^{2g+1}\alpha).  \nonumber
\end{align}
for some coefficients $c_{ik}$ which can be written down explicitly. For example, $c_{0,9} = 996  \zeta_q(2)^{-1} \cq , c_{1,8} = -2628  \zeta_q(2)^{-1} \cq, c_{2,7} = 2304  \zeta_q(2)^{-1} \cq, c_{3,6} = -672   \zeta_q(2)^{-1} \cq.$ Note that we can also write the expression above as a polynomial in $g$ and $m$, as in equation \eqref{secondexpr}. The coefficients $d_{nr}$ can be written down explicitly. For example, 
\begin{equation}
f_9:= d_{9,0} = -420 \cq, \label{f9}
\end{equation}
\begin{equation}
f_8 :=d_{8,0}=  828 \zeta_q(2)\cq  -10278  \cq +738  \cprimew + 828  \cprimex. \label{f8}
\end{equation}
 Similarly as for $T_{1,C_1}(V=0)$, when $n+r \leq 8$,  we get a term of size $O(g^7 \alpha)$, so we only consider the powers of $g$ greater than or equal to $8$. Using Lemma \ref{ak}, we compute a term of size $g^8 \alpha$, equal to $\frac{\alpha}{2 \cdot9!} g^8 f_8$. Now we can focus on those terms in \eqref{firstexpr} for which $i+k=9$. Then 
 \begin{align*}
T_{1,C_1}(V=\square) &= -q^{2g+1} \frac{\alpha}{2 \zeta_q(2) 9!} f_8 g^8 -q^{2g+1} \frac{1}{3! 9!}  \sum_{i=0}^3 c_{i, 9-i} \sum_{r=0}^i  \binom{i}{r}(-1)^r (2g)^{i-r} \sum_{e=0}^{9-i} (-1)^e g^{9-i-e} \binom{9-i}{e} \\
& \times  \Bigg[ \sum_{j=0}^{\alpha-1} \sum_{m=0}^{g} \frac{ \sin (2 \pi \theta(2m-j))}{2m-j} m^{r+e} +  \sum_{j=0}^{\alpha-1} \sum_{m=1}^{g} \frac{ \sin (2 \pi \theta(2m+j))}{2m+j} m^{r+e} \Bigg] +O(q^{2g+1}g^7 \theta_0^{-1} \alpha).
 \end{align*}
We use Lemma \ref{ak}, and consider the cases $e+r=0, e+r=1$ and $e+r \geq 2$ separately. When $e+r=1$, we get a term of size $g^8 \theta_0^{-1}$, and we can compute the coefficient of this term exactly from the coefficients $c_{i,9-i}$. We get
\begin{align*}
T_{1,C_1}(V=\square) &= -\frac{q^{2g+1}}{9! \zeta_q(2)} \Bigg[  \frac{\alpha}{2 } f_8 g^8 + \frac{ \alpha}{2 }  f_9 g^9  -  \frac{\alpha}{2}  (2g)^8  \cq \binom{9}{1}  \Bigg( \frac{\cos(2 \pi \theta_0)}{2 \pi \theta_0} \\
& + \sum_{m=1}^{g} m \sin( 4 \pi m \theta_0) (A(m)+B(m)) \Bigg)  -   \frac{\alpha \zeta_q(2)}{2 \cdot 3! 9!}  \sum_{i=0}^3 c_{i, 9-i} \sum_{r=0}^i  \binom{i}{r}(-1)^r (2g)^{i-r}\\
&\times \sum_{e=0}^{9-i} (-1)^e g^{9-i-e} \binom{9-i}{e} \Bigg( \frac{g^{e+r-1} \cos(4g \pi \theta_0)}{ 2 \pi \theta_0} - g^{e+r-1} \sin(4g \pi \theta_0) \Bigg) \Bigg] \\
& +O\Big(q^{2g+1} g^8 \theta_0 \alpha^3+ q^{2g+1}g^7 \theta_0^{-1} \alpha\Big).
\end{align*}
Since $\sum_{e=0}^{9-i} (-1)^e \binom{9-i}{e}=0$, the last term above cancels out.  This finishes the proof of Lemma \ref{compt1}.
   \end{proof}
   
  Now we go back to the proof of Theorem \ref{fourth}. We pick $\alpha= 100 \lfloor \log g \rfloor,$ and $\theta_0 = g^{-\frac{1}{2}}.$  Note that
  $$ \bq =  \cq = \prod_P \frac{(|P|-1)^6(|P|^5+7|P|^4-3|P|^3+6|P|^2-4|P|+1)}{|P|^{10}(|P|+1)}.$$
  Using equation \eqref{new}, and combining Lemmas \ref{maintrunc}, \ref{t2} and \ref{compt1}, we note that the terms of size $g^9 \alpha, g_8 \theta_0^{-1}, g^8 \alpha^2$ and $g^8 \alpha$ cancel out. We obtain an asymptotic formula with the $g^{10}$ and $g^9$ terms, and an error of size $q^{2g+1} g^{8+\frac{1}{2}+\epsilon} .$ We repeat this argument twice, which yields the asymptotic formula with the error of size $q^{2g+1} g^{7+\frac{1}{2}+\epsilon}$.

  \section{Upper bounds for moments of $L$-functions}
\label{upper}
Here, we prove Theorem ~\ref{upperbound}. The proof is very similar to Soundararajan's original bound on moments of the Riemann-zeta function. 
We first need the following.
\begin{lemma}
Let $\frac{1}{2} \leq \alpha \leq 1$ and $N$ be a positive integer. Then
$$ \log | L(\alpha+it,\chi_D)| \leq \frac{2g}{N+1} \log \Big( \frac{1+q^{-(\alpha-\frac{1}{2})(N+1)}}{1+q^{-2(N+1)}} \Big) + \Re \Big(\sum_{d(f) \leq N} \frac{ a_{\alpha}(d(f)) \chi_D(f) \Lambda(f)}{|f|^{\frac{1}{2} + it}} \Big) + O(1),$$
where the coefficient $a_{\alpha}(m)=0$ if $|m|>N$ and if $|m| \leq N$, it can be written down explicitly (see formula \eqref{coeff}). For $1 \leq |m| \leq N$, we have
$$ a_{\alpha}(m) = \frac{1}{|m| q^{|m|(\alpha-\frac{1}{2})}} -\frac{1}{|m|q^{2|m|}} +O \Big(  \frac{1}{(N+1) q^{(N+1)(\alpha-\frac{1}{2})}} \Big).$$ \label{lalfa}
\end{lemma}
The following is an easy corollary to the lemma above.
\begin{corollary}
We have the following bounds
$$ \log |L ( \tfrac{1}{2}+it,\chi_D)| \leq \frac{g \log 2}{\log_q g}+ O \Big(  \frac{g \log \log g}{(\log g)^2} \Big),$$ and for $\frac{1}{2}<\alpha \leq 1$, 
$$ \log |L(\alpha+it,\chi_D) | \ll \frac{g^{2-2\alpha}}{\log_q g}.$$
\end{corollary} \label{ubalfa}
\begin{proof}[Proof of Corollary]
When $\alpha= \frac{1}{2}$, pick $N = 2 \log_q g - 4 \log_q \log_q g$ and use Lemma \ref{lalfa}. When $\frac{1}{2} < \alpha\leq 1$, pick $N = 2 \log_q g$, and the conclusion follows.
\end{proof}
\begin{proof}[Proof of Lemma \ref{lalfa}]
We look at
$$ \Big| \frac{\Lambda(\alpha+it,\chi_D)}{\Lambda ( - \frac{3}{2}+it,\chi_D)} \Big| ,$$ and using the expression \eqref{completed} for $\Lambda(s,\chi_D)$ and the functional equation \eqref{fecompleted}, we get that
$$ | L (\alpha+it,\chi_D)| = q^{5g-2g \alpha}  \Big| L ( \tfrac{5}{2}+it,\chi_D) \Big| \prod_{j=1}^{2g} \Big(   \frac{q^{2 \alpha-1}+1-2 q^{\alpha- \frac{1}{2}} \cos( 2 \pi \theta_j- t \log q)}{q^4+1-2q^2 \cos(2 \pi \theta_j - t \log q)} \Big)^{\frac{1}{2}}.$$
Since $|L ( \frac{5}{2}+it,\chi_D)| \sim 1$ and 
$$q^{2 \alpha-1}+1-2 q^{\alpha- \frac{1}{2}} \cos( 2 \pi \theta_j- t \log q) = (q^{\alpha-\frac{1}{2}} - 1)^2 + 4q^{\alpha- \frac{1}{2}} \sin^2 ( \pi \theta_j - \frac{t \log q}{2} ),$$ with a similar expression holding for the denominator, it follows that
\begin{equation} \log | L (\alpha+it,\chi_D)| = g \Big( \frac{5}{2} - \alpha \Big) \log q - \frac{1}{2} \sum_{j=1}^{2g} \log \Big(  \frac{ a^2+\sin^2(\pi \theta_j - \frac{t \log q}{2})}{b^2+ \sin^2(\pi \theta_j - \frac{t \log q}{2})}\Big) +O(1), \label{eq1}
\end{equation} where
$$a= \frac{q^2-1}{2q}, \, b= \frac{q^{\alpha- \frac{1}{2}}}{2q^{ \frac{\alpha}{2} - \frac{1}{4}}}.$$ 
Now let $$f(x) = \log \Big(  \frac{a^2+ \sin^2(x)}{b^2+\sin^2(x)} \Big),$$ 
$f_1(x) = f( \pi x)-(\frac{5}{2}-\alpha) \log q$ and $f_2(x) = f_1(x- \frac{t \log q}{2 \pi }) .$ Then
\begin{equation} \log |L(\alpha+it,\chi_D)| = - \frac{1}{2} \sum_{j=1}^{2g} f_2( \theta_j) + O(1). \label{id}
\end{equation}
Similarly as in \cite{faicarneiro}, we want to find an appropriate minorant for $f_1$ (and hence for $f_2$) and then use the explicit formula in Lemma \ref{explicit}. We first compute the Fourier series for $f_1$.
\begin{equation} \label{fourierf1}
f_1(x) = \sum_{n \neq 0} \frac{e(nx)}{|n|} \Big( \frac{1}{q^{(\alpha-\frac{1}{2})|n|}}- \frac{1}{q^{2|n|}} \Big).
\end{equation}

We prove the following lemma, which describes the properties of the minorant for $f_1$.
\begin{lemma} \label{minorant}
Let $N$ be a positive integer. If $r$ is a real valued trigonometric polynomial of degree $N$ such that $r(x) \leq f_1(x)$ for all $x \in \mathbb{R}/\mathbb{Z}$, then 
$$ \int_{\mathbb{R}/\mathbb{Z}} r(x) \, dx \leq -\frac{2}{N+1} \log \Big( \frac{1+q^{-(N+1)(\alpha- \frac{1}{2})}}{1+q^{-2(N+1)}} \Big),$$ with equality if and only if $r(x) = \sum_{|n| \leq N} \hat{r}(n) e(nx)$, with
$\hat{r}(n) $ given in equations \eqref{rn} and \eqref{r0}.
\end{lemma}
\begin{proof}
The lemma follows quite easily by combining ideas from \cite{carneirovaaler}. Keeping the notation in \cite{carneirovaaler}, let $G_{\lambda}(x) = e^{- \pi \lambda x^2}$, with $\lambda>0$, and let
$$ L(\lambda,z) = \Big( \frac{\cos \pi z}{\pi} \Big)^2 \Big[  \sum_{m=-\infty}^{\infty} \frac{G_{\lambda}(m+ \frac{1}{2})}{(z-m-\frac{1}{2})^2} + \sum_{n=-\infty}^{\infty} \frac{G_{\lambda}'(n+\frac{1}{2})}{z-n-\frac{1}{2}}\Big].$$ For $x \in \mathbb{R}/\mathbb{Z}$, let
$$l(\lambda,N,x) = \lambda^{\frac{1}{2}} (N+1)^{-1} \sum_{|n| \leq N} \hat{L} \Big(  \frac{\lambda}{(N+1)^2}, \frac{n}{N+1}\Big) e(nx).$$
Now for $\tau$ a complex number with $\Im (\tau)>0$, let $q=e^{\pi i \tau}$ and
$$ \theta_1(v,\tau)= \sum_{n=-\infty}^{\infty} q^{(n+\frac{1}{2})^2}e ((n+\frac{1}{2})v),$$
$$\theta_2(v,\tau) = \sum_{n=-\infty}^{\infty} (-1)^n q^{n^2} e(nv),$$ 
$$ \theta_3(v,\tau)= \sum_{n=-\infty}^{\infty} q^{n^2} e(nv).$$
Next we define $p:(0,\infty) \times \mathbb{R}/\mathbb{Z} \to \mathbb{R}$ by
$$ p(\lambda,x) = - \lambda^{-\frac{1}{2}} + \lambda^{-\frac{1}{2}} \sum_n e(nx) e^{-\pi \lambda^{-1} n^2} = -\lambda^{-\frac{1}{2}} + \lambda^{-\frac{1}{2}} \theta_3(x,i \lambda^{-1}).$$ By Theorem $6$ in \cite{carneirovaaler}, if $q(x)$ is a real valued trigonometric polynomial of degree at most $N$ with $q(x) \leq p(\lambda,x)$ for all $x \in \mathbb{R}/\mathbb{Z}$, then 
$$ \int_{\mathbb{R}/\mathbb{Z}} q(x) \, dx \leq - \lambda^{-\frac{1}{2}}+ \lambda^{-\frac{1}{2}} \theta_2( 0 , i \lambda^{-1} (N+1)^2),$$ with equality if and only if 
$ q(x) = - \lambda^{-\frac{1}{2}} + \lambda^{-\frac{1}{2}} l(\lambda, N,x).$ Now let $\mu$ be the finite non-negative Borel measure on $(0,\infty)$ defined by 
$$ d \mu(\lambda) = \frac{e^{- \pi \lambda c^2} - e^{-\pi \lambda d^2}}{\lambda} \, d \lambda,$$ with $0<c<d$,
and let $h_{\mu}(x) = \int_0^{\infty} p (\lambda,x) \, d \mu(\lambda).$ Then we compute the Fourier series
\begin{equation}
h_{\mu}(x) = \sum_{n \neq 0} \frac{e(nx)}{|n|} \Big(\frac{1}{e^{2 \pi c |n|}} - \frac{1}{e^{2 \pi d |n|}} \Big).
\label{hmu}
\end{equation} Define
$$r_{\mu}(x) = \sum_{|n| \leq N} \hat{r}_{\mu}(N,n) e(nx),$$ with
\begin{equation}
\hat{r}_{\mu}(N,n) = \int_0^{\infty} \frac{1}{N+1} \hat{L} \Big( \frac{\lambda}{(N+1)^2}, \frac{n}{N+1} \Big) \, d\mu(\lambda) \label{rn}
\end{equation} for $n \neq 0$ and 
\begin{equation}
\hat{r}_{\mu}(N,0) = \int_0^{\infty} \Big[- \lambda^{-\frac{1}{2}}+ \lambda^{-\frac{1}{2}} \theta_2( 0 , i \lambda^{-1} (N+1)^2)  \Big]\, d\mu(\lambda). \label{r0}
\end{equation}
Using Theorem $6$ and the ideas in Corollary $17$  from \cite{carneirovaaler}, it follows that $r_{\mu}(N,x)$ is the optimal minorant for $h_{\mu}(x)$ (in the sense that if $q(x) \leq h_{\mu}(x)$ for all $x \in \mathbb{R}/\mathbb{Z}$, then $\int_{\mathbb{R}/\mathbb{Z}} q(x) \, dx \leq \int_{\mathbb{R}/\mathbb{Z}} r_{\mu}(N,x) \, dx$.)

Now we pick $c= \frac{ (\alpha-\frac{1}{2}) \log q}{2 \pi }$ and $d= \frac{ \log q}{\pi}$ and using \eqref{hmu} together with \eqref{fourierf1}, we have $h_{\mu}(x) = f_1(x)$. Now Lemma \ref{minorant} follows by noting that 
$$\hat{r}_{\mu}(N,0) = \int_0^{\infty} \Big[- \lambda^{-\frac{1}{2}}+ \lambda^{-\frac{1}{2}} \theta_2( 0 , i \lambda^{-1} (N+1)^2)  \Big]\, d\mu(\lambda) = -\frac{2}{N+1} \log \Big( \frac{1+q^{-(\alpha-\frac{1}{2})(N+1)}}{1+q^{-2(N+1)}} \Big).$$
\end{proof}

Now we return to the proof of Lemma \ref{lalfa}. Let $r_{\alpha}$ denote the optimal minorant found in Lemma \ref{minorant}. Using the definition of $f_2$, equation \eqref{id}, the explicit formula in Lemma \ref{explicit} and Lemma \ref{minorant}, it follows that
$$  \log | L(\alpha+it,\chi_D)| \leq \frac{2g}{N+1} \log \Big( \frac{1+q^{-(\alpha-\frac{1}{2})(N+1)}}{1+q^{-2(N+1)}} \Big) + \Re \Big( \sum_{d(f) \leq N} \frac{ \hat{r}_{\alpha}(d(f)) \chi_D(f) \Lambda(f)}{|f|^{\frac{1}{2} + it}} \Big)+O(1).$$
Now we use equation \eqref{rn} to write down $\hat{r}_{\alpha}(m)$ explicitly. We use the fact that for $|t| \leq 1$,
$$ \hat{L} (\lambda, t) = (1-|t|) \theta_1(t,i \lambda)-(2 \pi)^{-1} \lambda \, \text{sgn}(t) \frac{ \partial \theta_1}{\partial t} (t, i \lambda),$$ which is proven in Theorem $4$ in \cite{carneirovaaler}. Then for $|m| \leq N$, 
\begin{align*}
\hat{r}_{\alpha}(m) &= \frac{1}{N+1} \int_0^{\infty} \Big[  \Big( 1- \frac{|m|}{N+1} \Big) \sum_{n=-\infty}^{\infty} e^{- \pi \lambda \frac{(n+\frac{1}{2})^2 }{(N+1)^2}} e \Big((n+\tfrac{1}{2}) \frac{m}{N+1} \Big)  - \frac{\lambda}{2 \pi (N+1)^2} \, \text{sgn}(m) \\
& \times \sum_{n=-\infty}^{\infty} e^{-\pi \lambda  \frac{(n+\frac{1}{2})^2 }{(N+1)^2}} e \Big((n+\tfrac{1}{2}) \frac{m}{N+1} \Big)  2 \pi i (n+\tfrac{1}{2} ) \Big] \frac{ e^{- \frac{\lambda (\alpha-\frac{1}{2})^2 (\log q)^2}{4 \pi}} - e^{- \frac{ \lambda (\log q)^2}{\pi}}}{\lambda} \, d \lambda.
\end{align*}
Using the computations in section $4$ of \cite{faicarneiro} we get that for $|m| \leq N$, 
\begin{align}
\hat{r}_{\alpha}(m) &= \sum_{k=0}^{\infty} (-1)^k (k+1) \Big[ \frac{1}{|m|+k(N+1)} \Big( q^{-(\alpha-\tfrac{1}{2})(|m|+k(N+1))}-q^{-2(|m|+k(N+1))} \Big) \nonumber \\
&-\frac{1}{(N+1)(k+2) - |m|} \Big( q^{(\alpha-\tfrac{1}{2})(|m|-(k+2)(N+1))} - q^{2(|m|-(k+2)(N+1))} \Big) \Big]. \label{coeff}
\end{align}
Now let $a_{\alpha}(m)=\hat{r}_{\alpha}(m)$. Note that the first summand above is 
$$ \frac{1}{|m|} \Big( \frac{1}{q^{|m|(\alpha-\tfrac{1}{2})}}-\frac{1}{q^{2|m|}} \Big) + O \Big(\frac{1}{(N+1)q^{|m|(\alpha-\frac{1}{2})}q^{(N+1)(\alpha- \tfrac{1}{2})}} \Big),$$ while the second is bounded by $\frac{1}{(N+1)q^{(N+1)(\alpha-\frac{1}{2})}}.$ Combining these two bounds finishes the proof of Lemma \ref{lalfa}.
\end{proof}

Before the proof of Theorem \ref{upperbound}, we also need the following lemma, which is the analog of Lemma $6.3 $ in \cite{soundyoung}.
\begin{lemma}
Let $k,y$ be integers such that $2 ky \leq 2g+1$. For any complex numbers $a(P)$, we have
$$ \sum_{D \in \mathcal{H}_{2g+1}} \left| \sum_{d(P) \leq y} \frac{ \chi_D(P) a(P)}{ \sqrt{|P|}} \right|^{2k} \ll q^{2g+1} \frac{ (2k)!}{k! 2^k} \left( \sum_{d(P) \leq y} \frac{|a(P)|^2}{|P|} \right)^k.$$ \label{mom}
\end{lemma}
\begin{proof}
The proof is similar to the proof of Lemma $6.3$ in \cite{soundyoung} and uses the Polya-Vinogradov bound in Theorem \ref{pv}. 
\end{proof}
To prove Theorem ~\ref{upperbound}, we will need estimates for the frequency of large values of $|\mathcal{L}(u/\sqrt{q},\chi_D)|$. As $D$ varies over polynomials in $\mathcal{H}_{2g+1}$, we expect $\log \, |\mathcal{L}(u/\sqrt{q},\chi_D)|$ to be normally distributed with mean $\mathcal{M}(u,g)$ and variance $\mathcal{V}(u,g)$. Let
$$N(V,u,g) = \left| \{ D \in \mathcal{H}_{2g+1} \, | \log |\mathcal{L} ( u/\sqrt{q}, \chi_D )| \geq \mathcal{M}(u,g)+V \} \right|.$$ We will prove the following.
\begin{lemma}
With the same notation as above, if $\sqrt{\log g} \leq V \leq \mathcal{V}(u,g)$, then
$$N(V,u,g) \ll q^{2g+1} \exp  \left( -\frac{V^2}{2 \mathcal{V}(u,g)} \left( 1- \frac{8}{\log \log g} \right) \right);$$
if $\mathcal{V}(u,g) < V \leq \frac{\mathcal{V}(u,g)}{6} \log \log g$, then
$$ N(V,u,g) \ll  q^{2g+1} \exp  \left( -\frac{V^2}{2 \mathcal{V}(u,g)} \left( 1- \frac{4V}{\mathcal{V}(u,g) \log \log g} \right)^2 \right);$$
if $V > \frac{\mathcal{V}(u,g)}{6} \log \log g$, then
$$ N(V,u,g) \ll  q^{2g+1} \exp \left( - \frac{V \log V}{432} \right).$$
\label{frequency}
\end{lemma}
\begin{proof}[Proof of Theorem ~\ref{upperbound}]
Using the lemma above, we can prove the upper bound for the $k^{\text{th}}$ moment as follows. Note that
\begin{equation}
 \sum_{D \in \mathcal{H}_{2g+1}} | \mathcal{L}(u/\sqrt{q},\chi_D)|^k = \int_{-\infty}^{\infty} \exp ( kV + k \mathcal{M}(u,g)) d N(V,u,g) = k \int_{-\infty}^{\infty} \exp(kV+k \mathcal{M}(u,g)) N(V,u,g) \, dV.
 \label{obser}
 \end{equation}
  We use Lemma ~\ref{frequency} in the form
$$ N(V,u,g) \ll 
\begin{cases}
q^{2g+1} g^{o(1)} \exp \left( - \frac{V^2}{2 \mathcal{V}(u,g)} \right) & \mbox{ if } V \leq 8 k \mathcal{V}(u,g) \\
q^{2g+1} g^{o(1)} \exp(-4kV) & \mbox{ if } V > 8 k \mathcal{V}(u,g) .
\end{cases}
$$
Using these bounds in ~\eqref{obser}, we get the desired upper bound.
\end{proof}
\begin{proof}[Proof of Lemma ~\ref{frequency}]
For $N$ as in Lemma \ref{lalfa}, let $\frac{g}{N+1} = \frac{V}{A}$ and $N_0= \frac{N}{\log g},$ where 
$$ A = 
\begin{cases}
 \log \log g & \mbox{ if } V \leq \mathcal{V}(u,g) \\
 \frac{\mathcal{V}(u,g)}{V} \log \log g & \mbox{ if } \mathcal{V}(u,g) < V \leq \frac{\mathcal{V}(u,g)}{6} \log \log g \\
 6 & \mbox{ if } V > \frac{\mathcal{V}(u,g)}{6} \log \log g.
\end{cases}
$$
We use Proposition ~\ref{lalfa}, and then
$$ \log \Big| L ( \tfrac{u}{\sqrt{q}} ,
\chi_D) \Big| \leq \frac{2g}{N+1} \log 2 +  \sum_{d(f) \leq N} \frac{a_0(d(f)) \chi_D(f) \Lambda(f) \cos( \theta d(f))}{  \sqrt{|f|}} +O(1),$$
where $$a_0(d(f)) = \frac{1}{d(f)} - \frac{1}{d(f)|f|^2} + O \Big( \frac{1}{N+1}\Big).$$
Notice that for the sum over primes, the contribution from $f=P^r$ with $r \geq 3$ is bounded by $O(1)$. Using Lemma \ref{ubvar} and the Prime Polynomial Theorem in \eqref{ppt}, the contribution from $f=P^2$ is (up to an error of size $O( \log \log g)$ coming from those $P$ with $P|D$):
\begin{align*}
\sum_{d(P) \leq \frac{N}{2}} \frac{a_0(2 d(P)) d(P) \cos(2 \theta d(P))}{|P|} \leq \mathcal{M}(u,g)+\frac{g}{N+1}+O(1).
\end{align*}
Combining the above, it follows that
$$  \log \Big| L ( \frac{u}{\sqrt{q}} ,
\chi_D) \Big| \leq  \frac{3g}{N+1} + \mathcal{M}(u,g) + \sum_{d(P) \leq N} \frac{a_0(d(P)) \chi_D(P) d(P) \cos(\theta d(P))}{\sqrt{|P|}} .$$
If $D$ is such that $\log |\mathcal{L}(u/\sqrt{q},\chi_D)| \geq \mathcal{M}(u,g) + V$, then
$$  \sum_{d(P) \leq N} \frac{a_0(d(P)) \chi_D(P) d(P) \cos(\theta d(P))}{\sqrt{|P|}}   \geq V-\frac{3g}{N+1} = V \left(1- \frac{3}{A} \right).$$ Let $S_1$ be the sum above truncated at $d(P) \leq N_0$ and $S_2$ be the sum over primes $P$ with $N_0 < d(P) \leq N$. Then either $S_2 \geq \frac{V}{A}$ or $S_1 \geq V \left(1 - \frac{4}{A} \right):=V_1$. Let $\mathcal{F}_1 = \{ D \in \mathcal{H}_{2g+1} \, | S_1 \geq V_1 \}$ and $\mathcal{F}_2 = \{ D \in \mathcal{H}_{2g+1} \, | S_2 \geq V/A \}$.

If $D \in \mathcal{F}_2$, then using Lemma ~\ref{mom} and Markov's inequality, it follows that
$$|\mathcal{F}_2| \ll q^{2g+1} \Big( \frac{A}{V} \Big)^{2k}  \frac{ (2k)!}{k! 2^k} \Big( \sum_{ N_0< d(P) \leq N} \frac{|a(P)|^2}{|P|} \Big)^k,$$ for any $k$ such that $k \leq V/A$, where
$$a(P) = a_0(d(P)) d(P) \cos( \theta d(P))= \cos(\theta d(P)) - \frac{\cos(\theta d(P))}{|P|^2} + O \Big( \frac{d(P)}{N+1}\Big).$$ Using the fact that $|a(P)| \ll 1$ and picking $k = \left \lfloor{\frac{V}{A}}\right \rfloor  $,
\begin{equation}
|\mathcal{F}_2| \ll q^{2g+1} \left( \frac{A}{V} \right)^{2k} \left( \frac{2k}{e} \right)^k ( \log \log g)^k \ll q^{2g+1} \exp \left( - \frac{V}{2A} \log V \right). \label{measf2}
\end{equation}
If $D \in \mathcal{F}_1$, then for any $k \leq g/N_0$, 
$$|\mathcal{F}_1| \ll q^{2g+1} \frac{1}{V_1^{2k}}  \frac{ (2k)!}{k! 2^k} \left( \sum_{ d(P) \leq N_0} \frac{|a(P)|^2}{|P|} \right)^k,$$  with $a(P)$ as before. Now we use the Prime Polynomial Theorem \eqref{ppt}, the expression for $a(P)$ and Lemma ~\ref{ubvar} to get that
\begin{align*}
|\mathcal{F}_1|&  \ll q^{2g+1} \frac{1}{V_1^{2k}}   \left( \frac{2k}{e} \right)^k \left( \frac{1}{2} \sum_{n=1}^g \frac{1}{n} + \frac{1}{2} \sum_{n=1}^g \frac{\cos(2n \theta)}{n} \right)^k \ll q^{2g+1}   \left( \frac{2k \mathcal{V}(u,g)}{e V_1^2} \right)^k .
\end{align*}
If $V \leq \mathcal{V}(u,g)^2$, we pick $k =  \left \lfloor{\frac{V_1^2}{2 \mathcal{V}(u,g)}}\right \rfloor$ and if $V > \mathcal{V}(u,g)^2$, we pick $k = \left \lfloor{ 10V }\right \rfloor$.
Then 
\begin{equation}
 | \mathcal{F}_1 | \ll q^{2g+1} \exp \left( - \frac{V_1^2}{2 \mathcal{V}(u,g)} \right) + q^{2g+1} \exp (-4V \log V). 
 \label{measf1}
 \end{equation}
 Combining the bounds ~\eqref{measf2} and ~\eqref{measf1}, Proposition ~\ref{frequency} follows.
\end{proof}

Now we go back to the proof of Lemma ~\ref{shiftedupperb}, which is very similar to the proof of Theorem \ref{upperbound}, so we will skip some of the details.
\begin{proof}[Proof of Lemma ~\ref{shiftedupperb}]
Let $N(V)= | \{D \in \mathcal{H}_{2g+1} \, | \log | \mathcal{L}(w,\chi_D) \cdot \ldots \cdot \mathcal{L}(wu^{l-1},\chi_D) | \geq \mathcal{M}(w,g) + V \}|$.
Let $\mathcal{V} = \mathcal{V}(w,g) + \frac{3}{ \gamma^2}$. Notice that from Corollary \ref{ubalfa}, it is enough to consider $V \leq \frac{cg}{\log_q g},$ for some constant $c$.
We will prove the following.
If $V \leq \mathcal{V}$, then
\begin{equation}
N(V) \ll q^{2g+1} \exp \left( - \frac{V^2}{2 \mathcal{V}} \left(1-\frac{14}{\log \mathcal{V}} \right) \right);
\label{nv1}
\end{equation}
if $\mathcal{V} < V <  \frac{ \mathcal{V} \log \mathcal{V}}{8}$, then
\begin{equation}
N(V) \ll q^{2g+1} \exp \left( - \frac{V^2}{2 \mathcal{V}} \left( 1- \frac{7V}{ \mathcal{V} \log \mathcal{V}} \right)^2 \right);
\label{nv2}
\end{equation}
if $V >  \frac{ \mathcal{V} \log \mathcal{V}}{8}$, then
\begin{equation}
N(V) \ll q^{2g+1} \exp \left( - \frac{V \log V}{2048} \right).
\label{nv3}
\end{equation} Using these bounds, we get that
$$ \sum_{D \in \mathcal{H}_{2g+1}}   \log | \mathcal{L}(w,\chi_D) \cdot \ldots \cdot \mathcal{L}(wu^{l-1},\chi_D) |^k \ll q^{2g+1} g^{\epsilon} \exp \left( k ( \mathcal{M}(w,g)+ \tfrac{1}{\gamma}) + \frac{k^2}{2} (\mathcal{V}(w,g) + \tfrac{3}{\gamma^2} ) \right).$$ Since $1/\gamma^2 = o (\log g)$, Lemma ~\ref{shiftedupperb} follows.

Now to prove the bounds in ~\eqref{nv1}, ~\eqref{nv2} and ~\eqref{nv3}, write $w= \frac{1}{\sqrt{q}} e^{i \theta}$ and $u=\frac{1}{q^{\gamma}} e^{i \beta}$. Note that $\mathcal{L}(wu^j,\chi_d) = L(\tfrac{1}{2}+j \gamma-\frac{i(\theta+j \beta)}{\log q}).$ Using Lemma \ref{lalfa} for each of the $L$--functions, it follows that
\begin{align*}
& \log | \mathcal{L}(w,\chi_D) \cdot \ldots \cdot \mathcal{L}(wu^{l-1},\chi_D) | \leq \frac{2g}{N+1} \log 2 + \frac{2g}{N+1} \sum_{j=1}^{l-1} \log (1+q^{-(N+1) j \gamma})\\
& + \sum_{d(f) \leq N} \sum_{j=0}^{l-1} \frac{ a_j(d(f)) \chi_D(f) \Lambda(f)}{\sqrt{|f|}} \cos((\theta+j \beta)d(f)) + O (\gamma^{-1}),
\end{align*}
where for ease of notation, $a_j(m) = a_{\frac{1}{2}+j \gamma}(m)$ as in Lemma \ref{lalfa}.  Using the fact that $\log (1+x) < x$ for $x>0$, we get that
\begin{align} \label{ineq4}
\log | \mathcal{L}(w,\chi_D) \cdot \ldots \cdot \mathcal{L}(wu^{l-1},\chi_D) | &\leq \frac{2g}{N+1} (\log 2+1) +  \sum_{d(f) \leq N} \sum_{j=0}^{l-1} \frac{ a_j(d(f)) \chi_D(f) \Lambda(f)}{\sqrt{|f|}} \cos((\theta+j \beta)d(f))\\
& + O(\gamma^{-1}). \nonumber
\end{align}
Now let $\frac{g}{N+1}=\frac{V}{A}$, where $A$ is defined by
$$ A= 
\begin{cases}
\log \mathcal{V} & \mbox{ if } V \leq \mathcal{V} \\
\frac{\mathcal{V} \log \mathcal{V}}{V}  & \mbox{ if } \mathcal{V} < V \leq \frac{ \mathcal{V} \log \mathcal{V}}{8} \\
8 & \mbox{ if }  \frac{ \mathcal{V} \log \mathcal{V}}{8} < V,
\end{cases}
$$ and $\frac{N}{N_0} = \log g$. 

Using the asymptotics for $a_j(m)$ from Lemma \ref{lalfa}, note that the contribution from $f=P^k$ with $k \geq 3$ is of size $O(1)$. The contribution from square polynomials $f=P^2$ is equal to (up to an error of size $O(\gamma^{-1} \log \log g)$ coming from $P|D$)
\begin{align*}
\sum_{d(P) \leq \frac{N}{2}} \sum_{j=0}^{l-1} \frac{a_j(2d(P)) d(P)}{|P|} \cos((\theta+j \beta) 2 d(P)),
\end{align*} and again using the asymptotics for $a_j(m)$, it follows that the contribution from primes square is equal to
\begin{align*}
& \frac{1}{2} \sum_{d(P) \leq \frac{N}{2}} \frac{\cos(2 \theta d(P))}{|P|} + \frac{1}{2} \sum_{d(P) \leq \frac{N}{2}} \frac{1}{|P|} \sum_{j=1}^{l-1}\frac{\cos((\theta+j \beta)2d(P))}{|P|^{2 j \gamma}}+O(\gamma^{-1}) \\
&\leq \mathcal{M}(w,g) + \frac{g}{N+1} + \frac{1}{2}  \sum_{d(P) \leq \frac{N}{2}} \frac{1}{|P|} \sum_{j=1}^{\infty} \frac{1}{|P|^{2j \gamma}} +O(\gamma^{-1}) \leq \mathcal{M}(w,g) + \frac{g}{N+1}\\
& + \frac{1}{4 \gamma}  \sum_{d(P) \leq \frac{N}{2}} \frac{1}{|P|  d(P)}+O(\gamma^{-1})  = \mathcal{M}(w,g) + \frac{g}{N+1}  +O(\gamma^{-1}),
\end{align*} where the second inequality follows from the fact that $|P|^{2 \gamma-1} \geq 2 \gamma d(P),$ and the last identity follows by using the Prime Polynomial Theorem \eqref{ppt}. Using this in \eqref{ineq4} and since $1/\gamma^2 = o ( \log g)$, we get that
\begin{align}
\log | \mathcal{L}(w,\chi_D) \cdot \ldots \cdot \mathcal{L}(wu^{l-1},\chi_D) |& \leq \frac{6g}{N+1} +\mathcal{M}(w,g) \\
& + \sum_{d(P) \leq N} \sum_{j=0}^{l-1} \frac{a_j(d(P)) \chi_D(P) d(P)}{\sqrt{|P|}} \cos( (\theta+j \beta)d(P)) . \label{sump}
\end{align}
Let $S_1$ be the sum in \eqref{sump} truncated at $d(P) \leq N_0$ and let $S_2$ be the sum with $N_0 < d(P) \leq N$. 
If $D$ is such that $ \log | \mathcal{L}(w,\chi_D) \cdot \ldots \cdot \mathcal{L}(wu^{l-1},\chi_D) | \geq \mathcal{M}(w,g) + V $, then either $S_2 \geq \frac{V}{A} $ or $S_1 \geq V(1-\frac{7}{A}):=V_1$. Let $\mathcal{F}_2$ be the set of $D$ with $S_2 \geq \frac{V}{A}$ and $\mathcal{F}_1$ be the set of $D$ with $S_1 \geq V_1$. 
Similarly as in the proof of Theorem \ref{upperbound}, for any $k \leq g/N$, we have
$$ | \mathcal{F}_2| \ll q^{2g+1} \frac{1}{(V/A)^{2k}} \Big( \frac{2k}{e}\Big)^k \Big(  \sum_{ N_0 \leq d(P) \leq N} \frac{|a(P)|^2}{|P|}\Big)^k,$$
where $$a(P) = \sum_{j=0}^{l-1} a_j(d(P)) \cos((\theta+j \beta)d(P)) d(P).$$
Now using the expression for $a_j(m)$ given in Lemma \ref{lalfa} and using a similar upper bound as for the contribution from primes square, we have that
 $$a(P) \leq \cos( \theta d(P)) + \frac{3}{2\gamma d(P)} +O \Big( \frac{d(P)}{N} \Big) .$$
  We pick $k=\lfloor \frac{V}{A} \rfloor$ and then
  $$ | \mathcal{F}_2| \ll  q^{2g+1} \exp \left( \frac{V}{A} \log \Big(  \frac{2A}{eV} \Big( \log \log g + \frac{3V^2 (\log g)^2}{\gamma^2 g^2 A^2} \Big)\Big)\right).$$
  Using the fact that $V \leq \frac{cg}{\log_q g}$ and that $1/ \gamma^2 = o(\log g)$, we get that
 \begin{equation}
  | \mathcal{F}_2|  \ll q^{2g+1} \exp \left( - \frac{V}{2A} \log V \right). \label{bf2}
  \end{equation}
 If $S_1 \geq V_1$, for any $k \leq g/N_0$, 
 $$ | \mathcal{F}_1|  \ll q^{2g+1} \frac{1}{V_1^{2k} } \Big( \frac{2k}{e} \Big)^k \Big( \sum_{d(P) \leq N_0} \frac{|a(P)|^2}{|P|} \Big)^k,$$ with $a(P)$ as before.
   Using the Prime Polynomial Theorem \eqref{ppt}, we get that
 $$  \sum_{d(P) \leq N_0} \frac{|a(P)|^2}{|P|} \ll  \mathcal{V}(w,g) + \frac{3}{\gamma^2} =\mathcal{V}.$$ When $V \leq \mathcal{V}^2$, pick $k = \lfloor \frac{V_1^2}{2 \mathcal{V}} \rfloor$, and when $V > \mathcal{V}^2$, pick $k= \lfloor 10 V \rfloor$. Then
 \begin{equation} | \mathcal{F}_1| \ll q^{2g+1} \exp \left( - \frac{V_1^2}{2 \mathcal{V}} \right) + q^{2g+1} \exp(-4V \log V). \label{bf1}
 \end{equation}
  Using the bounds \eqref{bf1} and \eqref{bf2}, we obtain ~\eqref{nv1}, ~\eqref{nv2} and ~\eqref{nv3}, which finishes the proof of Lemma ~\ref{shiftedupperb}. 
\end{proof}

\section{Appendix}
\label{app}
\subsection{Sums involving $\sin$ and $\cos$}

We will prove the following auxiliary lemmas. 
\begin{lemma}
For $\theta \in [0, \pi)$, 
$$ \sum_{n=1}^g \frac{\cos(2n \theta)}{n} \leq \log \left( \min \Big\{g, \frac{1}{2 \theta} \Big\} \right) +O(1).$$ 
\label{ubvar}
\end{lemma}
\begin{proof}
If $g \leq \frac{1}{2 \theta}$, then 
$$ \sum_{n=1}^g \frac{\cos(2n \theta)}{n} \leq \sum_{n=1}^g \frac{1}{n}  = \log g + O(1).$$
If $g> \frac{1}{2 \theta}$, then 
$$  \sum_{n=1}^g \frac{\cos(2n \theta)}{n} = \sum_{n=1}^{\frac{1}{2 \theta}} \frac{\cos(2n \theta)}{n} + \sum_{\frac{1}{2 \theta}<n \leq g} \frac{\cos(2n \theta)}{n}.$$
For the first sum above, we have
$$  \sum_{n=1}^{\frac{1}{2 \theta}} \frac{\cos(2n \theta)}{n} \leq  \sum_{n=1}^{\frac{1}{2 \theta}} \frac{1}{n} = \log \Big( \frac{1}{2 \theta} \Big) + O(1).$$ For the second sum, we have
$$ \sum_{\frac{1}{2 \theta}<n \leq g} \frac{\cos(2n \theta)}{n} = O(1),$$ which follows by comparing the sum to the integral $ \int_{\frac{1}{2 \theta}}^{g} \frac{\cos(2t \theta)}{t} \, dt,$ and then using integration by parts. Putting everything together finishes the proof.
\end{proof}
\begin{lemma} \label{sumpowerssin}
Let $\theta$ be such that $1/\theta=o(g)$ and $k \geq 1$. Then
$$ \sum_{m=1}^{2g} m^k \sin(m \theta) = -\frac{(2g)^k \cos ( (2g+\frac{1}{2}) \theta)}{2\sin(\theta/2)}+ O \left(g^{k-1} \frac{1}{\sin^2 (\theta/2)}\right),$$ and
$$ \sum_{m=1}^{2g} m^k \cos ( m \theta) = \frac{(2g)^k \sin( (2g+\frac{1}{2}) \theta)}{2 \sin(\theta/2)} + O \left(g^{k-1} \frac{1}{\sin^2 (\theta/2)}\right).$$
\end{lemma}
\begin{proof}
Let $$f(\theta) = \sum_{m=1}^{2g} \sin(m \theta),$$ and $$h ( \theta) = \sum_{m=1}^{2g} \cos (m \theta).$$ We write $\sin( m \theta) = \frac{1}{2 i} (e^{im\theta}-e^{-i m \theta})$ and $\cos(m \theta) = \frac{1}{2} (e^{i m \theta}+ e^{- i m \theta}).$ Evaluating the geometric series and after some manipulations, we arrive at
\begin{equation}
 f(\theta) = \frac{ \cos( \theta/2) - \cos( (2g+\frac{1}{2}) \theta)}{2 \sin (\theta/2)}, \label{sumsin}
 \end{equation} and
\begin{equation}
h(\theta) = \frac{ \sin(\theta/2) + \sin ( (2g+ \frac{1}{2}) \theta)}{2 \sin (\theta/2)}. \label{sumcos}
\end{equation}
From the expression of $f$, 
$$ f^{(k)}(\theta) = 
\begin{cases}
\displaystyle \sum_{m=1}^{2g} m^k \sin ( m \theta) & \mbox{ if } k \equiv  0 \pmod 4 \\
 \displaystyle \sum_{m=1}^{2g} m^k \cos(m \theta) & \mbox{ if } k \equiv 1 \pmod 4 \\
 \displaystyle - \sum_{m=1}^{2g} m^k \sin ( m \theta) & \mbox{ if } k \equiv  2 \pmod 4 \\
 \displaystyle -  \sum_{m=1}^{2g} m^k \cos(m \theta) & \mbox{ if } k \equiv 3 \pmod 4.
\end{cases} $$ 
Also 
$$ h^{(k)}(\theta) = 
\begin{cases}
\displaystyle \sum_{m=1}^{2g} m^k \cos( m \theta) & \mbox{ if } k \equiv  0 \pmod 4 \\
- \displaystyle \sum_{m=1}^{2g} m^k \sin( m \theta) & \mbox{ if } k \equiv  1 \pmod 4 \\
- \displaystyle \sum_{m=1}^{2g} m^k \cos( m \theta) & \mbox{ if } k \equiv  2 \pmod 4 \\
\displaystyle \sum_{m=1}^{2g} m^k \sin( m \theta) & \mbox{ if } k \equiv  3 \pmod 4 .\\
\end{cases}
$$
By successively differentiating \eqref{sumsin} and \eqref{sumcos} and looking at the highest powers of $g$, Lemma \ref{sumpowerssin} follows.
\end{proof}
We also need the following.
\begin{lemma}
Let $\theta$ be such that $1/ \theta = o(g)$. Then 
$$ \sum_{ k=1}^{a-1} \frac{ \sin ( k \theta)}{k} = \frac{ \pi - \theta}{2} - \frac{ \cos(a \theta)}{2a \sin( \theta/2)} - \frac{ \sin(a \theta)}{2a}+ O \left( \frac{1}{a^2 \sin^2 (\theta/2)} \right) .$$
\label{truncatedseries}
\end{lemma}
\begin{proof}
We write 
$$ \sum_{k=1}^{a-1}  \frac{ \sin ( k \theta)}{k} = \sum_{k=1}^{\infty}  \frac{ \sin ( k \theta)}{k} - \sum_{k=a}^{\infty}  \frac{ \sin ( k \theta)}{k}.$$ One can show that 
\begin{equation}
 \sum_{k=1}^{\infty}  \frac{ \sin ( k \theta)}{k}  = \frac{ \pi - \theta}{2}.
 \label{expr}
 \end{equation}
  Now we need to evaluate the second sum above. Let 
$$ h (r,\theta) = \sum_{k=a}^{\infty} \frac{r^k}{k} \sin ( k \theta).$$ Then
\begin{align*}
\frac{ \partial}{ \partial r} h(r,\theta) &= \sum_{k=a}^{\infty} r^{k-1} \sin( k \theta)= \frac{1}{2 i} \sum_{k=a}^{\infty} r^{k-1} e^{ik \theta} - \frac{1}{2 i}  \sum_{k=a}^{\infty} r^{k-1} e^{- ik \theta} \\
&= \frac{r^{a-1}}{2 i} \left(  \frac{ e^{i a \theta} }{1- r e^{i \theta}} - \frac{ e^{-ia \theta}}{1-r e^{-i \theta}} \right) .
\end{align*}
We integrate the above with respect to $r$ from $0$ to $R$. Then
$$h(R,\theta) = \frac{1}{2 i} \left(  e^{i a \theta} \int_0^R \frac{r^{a-1}}{1-re^{i \theta}} \, dr - e^{-ia\theta} \int_0^R \frac{r^{a-1}}{1-re^{-ia \theta}} \, dr  \right).$$
Using integration by parts
$$  \int_0^R \frac{r^{a-1}}{1-re^{i \theta}} \, dr = \frac{R^a}{a(1-R e^{i \theta})} - \frac{e^{i \theta} R^{a+1}}{a(a+1)(1-Re^{i \theta})^2} + \frac{2e^{2 i \theta}}{a(a+1)} \int_0^R \frac{r^{a+1}}{(1-re^{i \theta})^3} \, dr ,$$ and similarly for the second integral. Now plugging in $R=1$, we have
\begin{equation}
h(1,\theta) = \sum_{k=a}^{\infty} \frac{\sin(k \theta)}{k} = \frac{\sin(a \theta)- \sin((a-1)\theta)}{a(2-2 \cos(\theta))} +O \left(\frac{1}{a^2 \sin^2(\theta/2)} \right)= \frac{\cos( (a-\frac{1}{2}) \theta) }{2 a\sin( \theta/2)}+ O \left(\frac{1}{a^2 \sin^2(\theta/2)} \right).
\label{tail}
\end{equation}
Combining equations \eqref{expr} and \eqref{tail} finishes the proof.
\end{proof}

\begin{lemma}
For $0 \leq k \leq 9$ and $1/\theta = o(1/g)$, we have
 \begin{align*} A(k,\theta) := \sum_{j=0}^{\alpha-1} \sum_{m=0}^{2g}  m^k  \frac{ \sin (2 \pi \theta(2m-j))}{2m-j}+\sum_{j=0}^{\alpha-1} \sum_{m=1}^{2g}  m^k  \frac{ \sin (2 \pi \theta(2m+j))}{2m+j} =  \end{align*}
 $$ \begin{cases}
 - \frac{\alpha}{2} \left[ \frac{ (2g)^{k-1} \cos(8g \pi \theta)}{2 \pi \theta} - (2g)^{k-1} \sin(8 g \pi \theta) \right] +O \Big( g^{k-1} \theta \alpha^3 + g^{k-2} \theta^{-1} \alpha \Big) & \mbox{ if } k \geq 2 \\
- \frac{\alpha}{2} \left[ \frac{\cos(8g \pi \theta)}{2 \pi \theta} - \sin(8g \pi \theta)  - \cot(2 \pi \theta)\right]  + \sum_{m=1}^{2g} m \sin(4 \pi m \theta) (A(m)+B(m))+ O(\theta \alpha^3) & \mbox{ if } k=1 \\
\frac{\pi \alpha}{2} - \frac{\alpha}{2} \left[   \frac{\cos(8g \pi \theta)}{(2g) 2 \pi \theta} - \frac{\sin(8g \pi \theta)}{2g} \right]   +O\Big(g^{-1} \theta \alpha^3 + g^{-2} \theta^{-2} \alpha \Big)  & \mbox{ if } k=0,
 \end{cases} $$
 where $A(m),B(m) = O (\frac{\alpha}{m^2})$ can be written down explicitly (see equation \eqref{aterm}.)
 \label{ak}
\end{lemma}
\begin{proof}
First assume that $k \geq 2$. Write $\sin(2 \pi \theta (2m+j)) = \sin(4 \pi m \theta) +2 \pi \theta j \cos(4 \pi m \theta)+ O(\theta^2 \alpha^2),$ by using the Taylor series for $\sin $ and $\cos$.  Note that the term of size $\theta^2 \alpha^2$ can be written down explicitly. Then using Lemma \ref{sumpowerssin}, this error term will be of size $O\Big( g^{k-1} \theta \alpha^3 \Big)$. Hence 
\begin{align*}
A(k,\theta) &= \sum_{m=1}^{2g} m^k \sin(4 \pi m \theta) \sum_{j=0}^{\alpha-1} \left(  \frac{1}{2m+j} +  \frac{1}{2m-j} \right) \\
&+ 2 \pi \theta \sum_{m=1}^{2g} m^k  \cos ( 4 \pi m \theta)  \sum_{j=0}^{\alpha-1} \left(  \frac{j}{2m+j} -  \frac{j}{2m-j} \right) + O \Big( g^{k-1} \theta \alpha^3 \Big).
\end{align*}
The sum over $j$ in the second term above is of size $\frac{\alpha^3}{m^2}$, so again using Lemma \ref{sumpowerssin}, the second term in the expression for $A(k,\theta)$ is bounded by $\alpha^3 g^{k-2}$. 
Using the asymptotic expansion for harmonic numbers, we have
$$ \sum_{j=0}^{\alpha-1} \frac{1}{2m+j} = \log \left( \frac{ 2m-1+\alpha}{2m-1} \right)  + \frac{1}{2(2m+\alpha-1)} - \frac{1}{2(2m-1)} - \sum_{k=1}^{\infty} \frac{B_{2k}}{2k} \left( \frac{1}{(2m+\alpha-1)^{2k}} - \frac{1}{(2m-1)^{2k}} \right),$$ where $B_{2k}$ are Bernoulli numbers.  Then we have
\begin{equation}
  \sum_{j=0}^{\alpha-1} \frac{1}{2m+j} = \frac{\alpha}{2m} + A(m),
  \label{firsteq}
  \end{equation}  where $A(m) = O(\frac{\alpha}{m^2}),$ and $A(m)$ can be written explicitly as
\begin{equation}
A (m)= \frac{\alpha}{2m(2m-1)} - \frac{\alpha}{2 ( 2m-1)(2m+\alpha-1)} + \sum_{k=2}^{\infty}  \frac{(-1)^{k+1} \alpha^k}{k (2m-1)^k} + \sum_{k=1}^{\infty} \frac{B_{2k}}{2k} \left( - \frac{1}{(2m+\alpha-1)^{2k}} + \frac{1}{(2m-1)^{2k}} \right). \label{aterm}
\end{equation}
Similarly we write
\begin{equation}
\sum_{j=0}^{\alpha-1} \frac{1}{2m-j} = \frac{\alpha}{2m} +B(m), \label{secondeq}
\end{equation} where $B(m)= O ( \frac{\alpha}{m^2})$ and $B(m)$ can be written down explicitly and has a similar expression as $A(m)$. When plugged into the expression for $A(k,\theta)$, the terms involving $A(m)$ and $B(m)$ will contribute a term of size $O\Big(g^{k-2} \theta^{-1} \alpha \Big)$ (by Lemma \ref{sumpowerssin}). 
Then 
$$ A(k,\theta) = \alpha \sum_{m=1}^{2g} m^{k-1} \sin(4 \pi m \theta) +  O\Big( g^{k-1} \theta \alpha^3 + g^{k-2} \theta^{-1} \alpha \Big)$$ Using Lemma \ref{sumpowerssin} and then again Taylor series for $\sin$ and $\cos$, we end up with
\begin{align}
 A(k,\theta)  &= - \alpha \frac{ (2g)^{k-1} \cos(8g \pi \theta)}{2 \sin(2 \pi \theta)} + \alpha \frac{(2g)^{k-1} \sin(8g \pi \theta)}{2} +O\Big( g^{k-1} \theta \alpha^2 + g^{k-2} \theta^{-1} \alpha \Big)  \nonumber \\
 &= - \frac{\alpha}{2} \left[ \frac{ (2g)^{k-1} \cos(8g \pi \theta)}{2 \pi \theta} - (2g)^{k-1} \sin(8 g \pi \theta) \right] +O \Big( g^{k-1} \theta \alpha^3 + g^{k-2} \theta^{-1} \alpha \Big).
 \end{align}
Now assume that $k=1$. We rewrite
$$A(1,\theta)= \sum_{j=0}^{\alpha-1} \sum_{m=0}^{2g}  m \left( \frac{ \sin (2 \pi \theta(2m-j))}{2m-j}+  \frac{ \sin (2 \pi \theta(2m+j))}{2m+j} \right).$$
We proceed similarly as when $k \geq 2$, and then 
$$A(1,\theta) = \sum_{j=0}^{\alpha-1} \sum_{m=1}^{2g} m \sin(4 \pi m \theta) \left( \frac{1}{2m+j}+\frac{1}{2m-j} \right) +O(\theta \alpha^3).$$
Using \eqref{firsteq} and \eqref{secondeq}, 
$$A(1,\theta) = \alpha \sum_{m=1}^{2g} \sin(4 \pi m \theta) + \sum_{m=1}^{2g} m \sin(4 \pi m \theta) (A(m)+B(m))+ O(\theta \alpha^3).$$
From the proof of Lemma \ref{sumpowerssin}, and again using Taylor series, we get that
\begin{align}
A(1,\theta) &= \alpha \frac{\cos(2 \pi \theta)- \cos(8g \pi \theta)}{2 \sin(2 \pi \theta)} + \alpha \frac{\sin(8 g \pi \theta)}{2} + \sum_{m=1}^{2g} m \sin(4 \pi m \theta) (A(m)+B(m))+ O(\theta \alpha^3) \nonumber \\
&=- \frac{\alpha}{2} \left[ \frac{\cos(8g \pi \theta)}{2 \pi \theta} - \frac{\cos(2 \pi \theta)}{\sin(2 \pi \theta)} - \sin(8g \pi \theta) \right]  + \sum_{m=1}^{2g} m \sin(4 \pi m \theta) (A(m)+B(m))+ O(\theta \alpha^3). \label{a1}
\end{align}
Now assume that $k=0$. We have
$$A(0,\theta) = \sum_{j=0}^{\alpha-1} \sum_{m=1}^{2g} \frac{\sin(2 \pi \theta(2m+j)}{2m+j} + \sum_{j=0}^{\alpha-1} \sum_{m=0}^{2g} \frac{\sin(2 \pi \theta(2m-j)}{2m-j} .$$ We  focus on the first summand above. Let $k = 2m+j$. We rewrite
\begin{align}
\sum_{j=0}^{\alpha-1} \sum_{m=1}^{2g} \frac{\sin(2 \pi \theta(2m+j)}{2m+j} &= \sum_{k=2}^{4g + \alpha-1} \frac{\sin(2 \pi \theta )k}{k} \sum_{m= \max\{ 1, \lceil \frac{k-\alpha+1}{2} \rceil \}}^{ \min\{2g, \lfloor \frac{k}{2} \rfloor \}} 1 =  \sum_{k=2}^{\alpha+1} \frac{\sin(2 \pi \theta k)}{k}  \floor[\Big]{ \frac{k}{2}}  \nonumber \\
&+ \frac{\alpha}{2} \sum_{k=\alpha+2}^{4g+1} \frac{\sin(2 \pi \theta k)}{k} + \sum_{k=4g+2}^{4g+\alpha-1} \frac{\sin(2 \pi \theta k)}{k}  \left(2g- \ceil[\Big]{ \frac{k-\alpha+1}{2}}+1 \right). \label{sum1}
\end{align}
Similarly 
\begin{align}
\sum_{j=0}^{\alpha-1} \sum_{m=0}^{2g} \frac{\sin(2 \pi \theta(2m-j)}{2m-j} &= \sum_{k=-(\alpha-1)}^{4g}  \frac{\sin(2 \pi \theta k)}{k}  \sum_{m = \max\{0, \ceil{\frac{k}{2}} \}}^{\min\{2g, \floor{ \frac{k+\alpha-1}{2}}\}}1 \nonumber \\
&= \sum_{k=-(\alpha-1)}^0  \frac{\sin(2 \pi \theta k)}{k}  \left( \floor[\Big]{ \frac{k+\alpha-1}{2} } +1 \right) + \frac{\alpha}{2} \sum_{k=1}^{4g-\alpha+2}  \frac{\sin(2 \pi \theta k)}{k} \nonumber \\
& + \sum_{k=4g-\alpha+3}^{4g}  \frac{\sin(2 \pi \theta k)}{k}  \left(2g- \ceil[\Big]{\frac{k}{2}}+1 \right) \label{sum2}
\end{align}
Combining equations \eqref{sum1} and \eqref{sum2}, we get that
\begin{align}
A(0,\theta) &= \frac{\alpha}{2} \sum_{k=0}^{4g+\alpha-1}  \frac{\sin(2 \pi \theta k)}{k}  + \frac{\alpha}{2} \sum_{k=1}^{4g-\alpha+2} \frac{\sin(2 \pi \theta k)}{k}  \label{b1} \\
&+ \sum_{k=4g-\alpha+3}^{4g} \frac{\sin(2 \pi \theta k)}{k}  \left( 2g - \frac{k-1}{2} \right)+ \sum_{k=4g+2}^{4g+\alpha-1} \frac{\sin(2 \pi \theta k)}{k}  \left( 2g - \frac{k-1}{2} \right) \label{b2} \\
&+ \frac{1}{2} \sum_{\substack{k=4g-\alpha+3 \\ k \text{ even}}}^{4g}  \frac{\sin(2 \pi \theta k)}{k}   - \frac{1}{2} \sum_{\substack{k=4g+2 \\ k \text{ even}}}^{4g+\alpha-1}  \frac{\sin(2 \pi \theta k)}{k} . \label{b3}
\end{align} Let $B_1$ denote the term \eqref{b1}, $B_2$ the term \eqref{b2} and $B_3$ the term \eqref{b3}. 
Using Lemma \ref{truncatedseries}, we have
$$   \sum_{k=0}^{4g+\alpha-1}  \frac{\sin(2 \pi \theta k)}{k} = \frac{ \pi +2 \pi \theta}{2} - \frac{ \cos((4g+\alpha)2 \pi \theta)}{2(4g+\alpha) \sin( \pi \theta)} - \frac{ \sin((4g+\alpha) 2 \pi \theta)}{2(4g+\alpha)} + O ( g^{-2}\theta^{-2}  ),$$
and
$$ \sum_{k=1}^{4g+2-\alpha} \frac{\sin(2 \pi \theta k)}{k}  = \frac{ \pi - 2 \pi \theta}{2} - \frac{\cos((4g-\alpha+3)2 \pi \theta)}{2(4g-\alpha+3) \sin( \pi \theta)} - \frac{\sin((4g-\alpha+3) 2 \pi \theta)}{2(4g-\alpha+3)} +  O ( g^{-2}\theta^{-2}  ).$$
We use the two equations above and express the $\sin$ and $\cos$ in terms of $\sin(8g \pi \theta)$ and $\cos(8g \pi \theta)$, and then use Taylor series. We get
\begin{align}
B_1 &= \frac{\alpha}{2} \left[ \pi - \frac{ \cos(8g \pi \theta)}{4g \sin( \pi \theta)} + \frac{2 \sin(8g \pi \theta)}{4g} \right]+ O \Big(g^{-1} \theta \alpha^3 + g^{-2} \theta^{-2} \alpha\Big) \nonumber \\
&= \frac{\pi \alpha}{2} - \frac{\alpha}{2} \left[   \frac{\cos(8g \pi \theta)}{(2g) 2 \pi \theta} - \frac{\sin(8g \pi \theta)}{2g} \right] + O \Big(g^{-1} \theta \alpha^3 + g^{-2} \theta^{-2} \alpha \Big) .
\label{b1f}
\end{align}
Now we evaluate $B_2$. By making a change of variables in each of the summands of $B_2$, we get
\begin{align}
B_2 &= \frac{1}{2} \sum_{y=1}^{\alpha-2} y \left( \frac{\sin(2 \pi \theta(4g+1-y))}{4g+1-y}- \frac{\sin(2 \pi \theta(4g+1+y))}{4g+1+y} \right) \nonumber \\
&= \sum_{y=1}^{\alpha-2} y \frac{ y \sin(2 \pi \theta(4g+1)) \cos(2 \pi y \theta) - (4g+1) \sin(2 \pi y \theta) \cos(2 \pi (4g+1) \theta)}{(4g+1)^2-y^2} =O (g^{-1} \theta \alpha^3). \label{b2estimate}
\end{align}

Now we look at $B_3$. We rewrite
\begin{align}
B_3 &= \frac{1}{4} \sum_{k=2g-\frac{\alpha}{2}+2}^{2g} \frac{\sin(4 \pi  \theta k ) }{k} - \frac{1}{4} \sum_{k=2g+1}^{2g+\frac{\alpha}{2}-1} \frac{\sin(4 \pi  \theta k )}{k} \nonumber \\
&=  \frac{1}{4} \sum_{k=2g-\frac{\alpha}{2}+2}^{2g} \frac{\sin(4 \pi  \theta k ) }{k}  - \frac{1}{4} \sum_{k=2g-\frac{\alpha}{2}+2}^{2g} \frac{\sin(4 \pi \theta(k+\frac{\alpha}{2}-1))}{k+\frac{\alpha}{2}-1} \nonumber \\
&= \frac{\frac{\alpha}{2}-1}{4} \sum_{k=2g+2-\frac{\alpha}{2}}^{2g} \frac{ \sin(4 \pi  \theta k )}{k(k-1+ \frac{\alpha}{2})} + O\Big(g^{-1} \theta \alpha^2 \Big)=  O \Big(g^{-2} \alpha^{2} + g^{-1} \theta \alpha^2 \Big) = O \Big( g^{-1} \theta \alpha^2 \Big)  . \label{b3estimate}
\end{align}
Combining \eqref{b1f}, \eqref{b2estimate} and \eqref{b3estimate}, it follows that
\begin{equation} A(0,\theta) = \frac{\pi \alpha}{2} - \frac{\alpha}{2} \left[   \frac{\cos(8g \pi \theta)}{(2g) 2 \pi \theta} - \frac{\sin(8g \pi \theta)}{2g} \right]   + O\Big(g^{-1} \theta \alpha^3 + g^{-2} \theta^{-2} \alpha \Big)  . \label{a0}
\end{equation}

\end{proof}
\subsection{Explicit formulas for the coefficients in Theorem \ref{fourth}}
By directly computing the residues in sections \ref{maintermsec} and \ref{secmain}, we find that
\begin{equation}
a_{10}=\frac{1}{\zeta_q(2)} \Bigg[ \frac{2048}{10!}  \bq -\frac{7680}{3!10!} \cq \Bigg],
\label{coeff1}
\end{equation}
\begin{align}
a_9 = \frac{1}{\zeta_q(2)} \Bigg[ & \frac{1}{10!} \Big(  51200 \bq - \frac{ \bprime}{q} \Big)  - \frac{1}{3!10!} \Big( -24000 \frac{q \cq}{q-1}   +216000 \cq - 13200 \cprimew  \nonumber \\
&-24000 \cprimex  \Big) \Bigg],
\label{coeff2}
\end{align}
and
\begin{align}
a_8 &= \frac{1}{\zeta_q(2)} \Bigg[ \frac{1}{10!} \Big( 560640 \bq - \frac{207360 \bprime}{q} + \frac{23040 \bdoubleprime}{q^2} \Big) - \frac{1}{3! 10!} \Big( -531360 \frac{q \cq}{q-1} +2616480 \cq  \nonumber\\
&- 347760 \cprimew +\frac{58320  \frac{d}{dw} \mathcal{C} (1,w) \rvert_{w=\frac{1}{q}} }{q-1}-17280 \frac{q \cprimex}{q-1} -531360 \cprimex \nonumber \\
&  +7560 \frac{  \frac{d^2}{dw^2} \mathcal{C} (1,w) \rvert_{w=\frac{1}{q}} }{q^2} +58320 \frac{  \frac{d}{dx} \frac{d}{dw} \mathcal{C}(x,w) \rvert_{w=\frac{1}{q},x=1} }{q} -8640  \frac{d^2}{dx^2} \mathcal{C} (x,\frac{1}{q}) \rvert_{x=1} \Big)  \Bigg] .
\label{coeff3}
 \end{align}
We will show (by direct computation) that these coefficients match the answer predicted by Andrade and Keating \cite{conjectures}. 
In \eqref{az}, denote by $A_j$ the partial derivative, evaluated at zero, of the function $A(z_1,\ldots, z_4)$ with respect to the $j^{\text{th}}$ variable. For ease of notation, we let $A_j(0,0,0,0)=A_j$ and $A(0,0,0,0)=A$. Let $Q(x) = \sum_{i=0}^{10} b_ix^i$. By directly computing the residues in the integral expression for $Q$ (in \eqref{qx}), we get that
$$b_{10} = \frac{A}{4725 \zeta_q(2)},$$
$$b_9 = \frac{1}{1890 \zeta_q(2)} \Big[ 10 A + \frac{1}{ \log q} \Big( A_1 +A_2 +A_3 +A_4 \Big)  \Big],$$ and
\begin{align*}
b_8 &= \frac{1}{1260 \zeta_q(2)} \Big[ 74 A + \frac{15}{\log g} \Big(A_1 +A_2 +A_3 +A_4 \Big) + \frac{2}{ ( \log q)^2 } \Big( A_{12} +  A_{13}+A_{14}+A_{23}+A_{24}+A_{34}  \Big)\Big].
\end{align*}
The fact that $a_{10}=b_{10}$ follows from the identity
$$\bq = \cq = A =  \prod_P \frac{(|P|-1)^6(|P|^5+7|P|^4-3|P|^3+6|P|^2-4|P|+1)}{|P|^{10}(|P|+1)}.$$
Now write $A_1=A_2=A_3=A_4 = aA,$ where
$$a = \sum_P \frac{d(P) (25|P|^4-16|P|^3+30|P|^2-20|P|+5)}{(|P|-1)(|P|^5+7|P|^4-3|P|^3+6|P|^2-4|P|+1)}.$$
We compute that
$$ \bprime = -2qa A, \,   \cprimew =-4qaA , \,   \cprimex = A \Big(-a- \sum_P \frac{d(P)}{|P|^2-1} \Big)= A \Big( -a- \frac{1}{q-1} \Big),$$ where the sum over primes above can be easily computed by looking at the logarithmic derivative of $\zeta_q(s)$. Combining the above identities gives that $a_9=b_9$. 
We further compute $A_{ij} = A(a^2+h),$ for $i \neq j$, where
$$ h= \sum_P -\frac{d(P)^2 |P| \left(17 |P|^9+26 |P|^8+13 |P|^7+57 |P|^6-117 |P|^5+113 |P|^4-65 |P|^3+27 |P|^2-8 |P|+1\right)}{(|P|-1)^2 \left(|P|^5+7 |P|^4-3 |P|^3+6 |P|^2-4 |P|+1\right)^2}.$$
We have $ \bdoubleprime = q^2 A (4a^2+2a-2b)$, where
$$b= \sum_P \frac{d(P)^2 \left(45 |P|^{10}+117 |P|^9-73 |P|^8+330 |P|^7-485 |P|^6+450 |P|^5-295 |P|^4+138 |P|^3-40 |P|^2+5 |P|\right)}{(|P|-1)^2 \left(|P|^5+7 |P|^4-3 |P|^3+6 |P|^2-4 |P|+1\right)^2}.$$
Also $$ \frac{d^2}{dw^2} \mathcal{C} (1,w) \rvert_{w=\frac{1}{q}} = q^2 A (16 a^2- 4(e-a)) ,  \, \, \frac{d^2}{dx^2} \mathcal{C} (x,\frac{1}{q}) \rvert_{x=1} = A \Bigg( \Big( a+ \frac{1}{q-1} \Big)^2 - \Big(r-a-\frac{1}{q-1} \Big) \Bigg),$$
$$  \frac{d}{dx} \frac{d}{dw} \mathcal{C}(x,w) \rvert_{w=\frac{1}{q},x=1} = qA \Bigg( 4a \Big( a+ \frac{1}{q-1} \Big) - 4f \Bigg),$$
where
$$ e =\frac{d(P)^2 \left(90 |P|^{10}+234 |P|^9-146 |P|^8+660 |P|^7-970 |P|^6+900 |P|^5-590 |P|^4+276 |P|^3-80 |P|^2+10 |P|\right)}{(|P|-1)^2 \left(|P|^5+7 |P|^4-3 |P|^3+6 |P|^2-4 |P|+1\right)^2},$$
\begin{align*}
 r & = \Big[ d(P)^2  \Big(38 |P|^{12}+220 |P|^{11}+123 |P|^{10}+305 |P|^9+89 |P|^8-98 |P|^7+34 |P|^6+20 |P|^5-89 |P|^4 \\
 &+98 |P|^3-43 |P|^2+7|P| \Big) \Big] \Big/ \Big[ (|P|-1)^2 (|P|+1)^2 \left(|P|^5+7 |P|^4-3 |P|^3+6 |P|^2-4 |P|+1\right)^2 \Big],
 \end{align*}
$$f=\frac{d(P)^2 |P| \left(28 |P|^9+91 |P|^8-86 |P|^7+273 |P|^6-368 |P|^5+337 |P|^4-230 |P|^3+111 |P|^2-32 |P|+4\right)}{(|P|-1)^2 \left(|P|^5+7 |P|^4-3 |P|^3+6 |P|^2-4 |P|+1\right)^2}. $$
The fact that $a_8=b_8$ follows from the above identities and upon noticing that
$$ \frac{7e}{2} +27f-r-24h-32b = \sum_P \frac{d(P)^2 |P|^2}{(|P|^2-1)^2} = \frac{q}{(q-1)^2},$$ where the last identity is obtained by looking at the second derivative of $\log \zeta_q(s)$.
\newline
\newline
 \textsc{Acknowledgments.} I would like to thank my advisor, K. Soundararajan, for the many helpful discussions we have had while I was working on this problem. I would also like to thank J. Keating and Z. Rudnick for their comments.
\bibliography{4th_bibl}
\bibliographystyle{plain}
\end{document}